\documentclass[abstracton=true, paper=a4, fontsize=11pt,DIV=12,bibliography=totoc, final]{scrartcl}

\pdfoutput=1

\usepackage[utf8]{inputenc}
\usepackage[T1]{fontenc}
\usepackage{lmodern}
\usepackage[english]{babel}
\usepackage{csquotes}
\usepackage{amsthm, amssymb, amsmath, amsfonts, mathrsfs, dsfont, esint,textcomp}
\usepackage[colorlinks=true, pdfstartview=FitV, linkcolor=blue, citecolor=blue, urlcolor=blue,pagebackref=false]{hyperref}
\usepackage{mathtools}
\usepackage[shortlabels]{enumitem}

\usepackage[backend=biber,giveninits,maxnames=4,  maxalphanames=5, style=alphabetic,isbn=false, date=year, doi=false, eprint= true, url= false, sorting=ynt, sortcites = true]{biblatex}

 \usepackage{authblk}
  
\setkomafont{date}{\large}

\addtokomafont{disposition}{\rmfamily}

\usepackage{microtype}

\usepackage[notcite,notref,color]{showkeys}
\definecolor{labelkey}{gray}{.8}
\definecolor{refkey}{gray}{.8}

\definecolor{darkblue}{rgb}{0,0,0.7} 
\definecolor{darkred}{rgb}{0.9,0.1,0.1}
\definecolor{darkgreen}{rgb}{0,0.5,0}

\renewenvironment{proof}[1][\smallskip\noindent\proofname]{{\smallskip\noindent\bfseries #1. }}{\qed \medskip}

\newtheorem{thm}{Theorem}[section]
\newtheorem{theorem}[thm]{Theorem}
\newtheorem{prop}[thm]{Proposition}
\newtheorem{lem}[thm]{Lemma}

\theoremstyle{remark}

\newtheorem{rem}[thm]{Remark}
\theoremstyle{definition}

 \renewcommand{\skew}{\mathrm{skew}}

\renewcommand{\leq}{\leqslant}
\renewcommand{\geq}{\geqslant}

\renewcommand{\subset}{\subseteq}

\newcommand{\B}{\mathcal{B}}
\newcommand{\E}{\mathbb{E}}

\newcommand{\mP}{\mathcal{P}}

\newcommand{\N}{\mathbb{N}}

\newcommand{\1}{\mathbf{1}}
\newcommand{\R}{\mathbb{R}}

\renewcommand{\P}{\mathbb{P}}
\renewcommand{\S}{\mathbb{S}}

\newcommand{\eps}{\varepsilon}
\newcommand{\e}{\varepsilon}
\renewcommand{\d}{{\mathrm{d}}}

\renewcommand{\O}{\mathcal{O}}

\newcommand{\dd}{\, \mathrm{d}}

\newcommand{\dmin}{d_{\mathrm{min}}}

\newcommand{\mR}{\mathcal R}
\newcommand{\mB}{\mathcal B}

\renewcommand{\varrho}{\rho}
\newcommand{\spinstate}{\{\pm1\}}

\DeclareMathOperator{\sym}{sym}

\DeclareMathOperator{\curl}{curl}

\DeclareMathOperator{\Id}{Id}

\DeclareMathOperator{\dv}{div}

\DeclareMathOperator*{\supp}{supp}

\DeclareMathOperator{\Bog}{Bog}

\DeclareMathOperator{\diam}{diam}

\DeclareMathOperator{\Cov}{Cov}

\numberwithin{equation}{section}

\mathtoolsset{showonlyrefs}

\bibliography{bib.bib}

\begin{document}

\title{\LARGE Derivation of a Multiscale Ferrofluid Model: Superparamagnetic Behavior due to Fast Spin Flip}
\author[1]{Alexandre Girodoux-Lavigne\thanks{girodroux@imj-prg.fr}}
\author[1]{Richard M. H\"ofer\thanks{richard.hoefer@ur.de}}
\affil[1]{Faculty of Mathematics, University of Regensburg, Germany}

\maketitle

\begin{abstract}
   We consider a microscopic model of $N$  magnetic nanoparticles in a Stokes flow. We assume that the temperature is above the critical N\'eel temperature such that the particles' magnetizations undergo random flip with rate $1/\eps$. The  microscopic system is the modeled through a piecewise deterministic Markov jump process.
      We show that for large $N$, small particle volume fraction and small $\eps$, the system can be effectively described by a multiscale model.
\end{abstract}

\tableofcontents

\section{Introduction}
Ferrofluids are artificial suspensions of magnetic nanoparticles in incompressible  liquids. They can be transported in a controlled way through externally applied magnetic fields.
This effect has various technological and medical applications ranging from cooling in loudspeakers to magnetic drug targeting. 

Ferrofluids are superparamagnetic: the magnetization increases when the external magnetic field is increased. Moreover, once the external magnetic field is turned off, the effective magnetization fades rapidly. This effect results either from  rotational Brownian motion that quickly relaxes the ordering of the particle orientation, or from spontaneous spin flip or reorientation that changes the magnetic dipoles without changing the physical orientation of the particles. Which of these relaxation mechanisms is dominant depends on the temperature and on  microscopic parameters of the ferrofluid.

Ferrofluids, like other suspensions such as dilute polymer solutions,  active suspensions or liquid crystals, can be modeled on different scales: either by macroscopic  fluid models,
or by so-called multiscale kinetic models, or else on the microscopic scale as suspended particles moving in a fluid flow.
\begin{itemize}
\item 
A typical \emph{macroscopic} model consists of the incompressible Navier--Stokes equations, with sources and stresses that are coupled to the particle concentration and -- in the case of magnetic particles -- to the magnetization. In macroscopic ferrofluid models, the  magnetization plays a similar role as the director field in the Leslie-Ericksen model for liquid crystals.  Various phenomenological macroscopic models  have been proposed and analyzed in the literature of ferrofluids.

Many models assume the particle distribution to be homogeneous in space (and consequently constant in time). For those ferrofluids, macroscopic PDE models have been proposed by Shliomis \cite{Shliomis02} and Rosensweig \cite{Rosensweig02}. Both models couple evolution equations for momentum and magnetization to Maxwell's equations or to their simplifications from magnetostatics. One can distinguish models where the internal magnetization $M$ is assumed to be parallel to the magnetic field and models where $M$ is determined though a PDE (so called models with \enquote{internal rotation}). Mathematically, these models have been analyzed for instance in \cite{AmiratHamdache08, AmiratHamdache16, NochettoSalgadoTomas16}.

Macroscopic models with inhomogeneous particle density have  been suggested and investigated. In one approach, the magnetization is a function of density and the external magnetic field \cite{PolevikovTobiska08, HimmelsbachNeuss-RaduNeuss18}. In a second approach,  Onsager's variational principle is used  to derive a thermodynamically consistent model which couples PDEs for the fluid velocity, particle density, magnetic field and magnetization \cite{GruenWeiss19, GruenWeiss21}.

\item
  So-called \emph{multiscale} or \emph{micro-macro} models go one step away from the pure phenomenology towards a microscopically more accurate description of particle suspensions. 
The macroscopic fluid is then coupled via the stresses and sources to a kinetic evolution equation for the  suspended solid phase.
The latter is modeled by a particle density function $f(t,x,\xi) \in \R^+$ at time $t$, where~$x$ is the position of particles and~$\xi \in \mathbb S^2$ is their (physical or magnetic) orientation. 
Such models describe how the microscopic state of the particles adapts collectively to local fluid deformations and how the macroscopic fluid flow gets itself effectively impacted.
Popular models include the kinetic FENE and Hookean dumbbell models for dilute suspensions of flexible polymers, the  Doi-(Onsager) model for suspensions of Brownian rigid rod-like particles,
and the Doi--Saintillan--Shelley for corresponding active (self-propelled) particles.
A multiscale model for ferrofluids with space-time homogeneous particle density has been proposed in \cite{ShenDoi90}. The  magnetization of the nanoparticles is modeled by a Fokker-Planck equation which is coupled to the momentum equation of the fluid.  
A related model for magnetoactive suspensions (magnetic bacteria) has been proposed in \cite{alonso2018microfluidic}. 
The models in \cite{ShenDoi90} and \cite{alonso2018microfluidic} account for Brownian motion but not for spin flip.
\smallskip\item 
At the  \emph{microscopic}  particle scale (more precisely the nanoscale), one can formulate a fully detailed hydrodynamic model describing the motion of suspended particles in a background fluid flow. The fluid phase is then described by the Navier--Stokes equations with no-slip conditions at the boundary of the particles and is coupled to the particle dynamics driven by Newton's equations of motion. This leads to a highly complex dynamics. 
\end{itemize}

\paragraph{Summary of the main result}
The aim of the present paper is a mathematical rigorous  derivation of an effective multiscale ferrofluid model starting from a microscopic description. 
Specifically, we are interested in the rigorous derivation of superparamagnetic behavior due to spontaneous spin flip of the nanoparticles, which is the dominant relaxation mechanism for sufficiently small particles (below the so-called Shliomis diameter)  or sufficiently high temperature (above the N\'eel temperature).

To this end, we consider a simplified model where we keep only the most relevant aspects that lead to this superparamagnetic behavior and to effective terms in the fluid equation. We refer to Section \ref{subsection Microscopic problem} for the precise description of the model. We model the fluid flow by the incompressible Stokes equations in $\R^3$ outside of $N$ identical rotationally symmetric particles where we impose  no-slip boundary conditions. The inertia of the particles is neglected such that the particle translational and angular velocities are determined through the fluid PDE and balance equations for the torque and force on each particle.
We neglect Brownian forces and torques as well as the internal magnetic field that one would obtain through the solution of magnetostatic equations. Hence the forces and torques on the particles are the sum of viscous ones and those determined by the external magnetic field, the particle orientation and its spin state.
As another simplification, we artificially freeze the time evolution of the particle centers and only allow them to rotate. 
We then couple the (piecewise) deterministic evolution of the particle orientations  to stochastic jumps of the particle spins in $\spinstate$ that occur with rate of order $1/\eps$ for each particle. Spin flips towards the energetically preferred spin that reduce the angle to $H$ occur more frequently than towards the opposite direction. 
The whole process is given in terms of a piecewise deterministic Markov jump process. 

We then show that for large $N$,  small $\eps$ (i.e. high spin flip rate) and small particle volume fraction $\phi$, the fluid velocity and the empirical particle density can be effectively described by a multiscale model. 
The effective fluid velocity solves the Stokes equations with a particle density dependent Kelvin force and stress term of order $\phi$. 
On the other hand, the particle density is in local equilibrium regarding the spins, i.e. the average spin of particles with orientation close to $\zeta$ and position close to $x$ is given by $\tanh \big( b \, H(t,x) \cdot \zeta \big)$ for some parameter $b > 0$.
Regarding the orientations,  the particles behave -- to leading order in $\phi$ -- as if each particle was isolated in the fluid with a magnetic torque acting on it and if their spin was the average spin at equilibrium. 
The precise result, including quantitative estimates, is stated in Theorem \ref{main theorem}. We further discuss the result, tis limitations and  possible generalizations after the statement of Theorem \ref{main theorem}.

\paragraph{Comparison with the literature}

The mathematical literature on the derivation of effective equations for suspensions is vast. We only discuss here some results that concern suspensions where the particle inertia is neglected and where the carrier fluid is non-magnetic.

 Suspensions of force-free but not torque-free particles are considered in \cite{Batchelor70torque}. The particle positions as well as the torques are prescribed. The author then formally derives a formula for the effective stress for dilute suspensions. 
 In \cite{LevyHsieh88}, the homogenization of the Navier-Stokes equations outside of periodically distributed magnetic particles is considered. The particle positions and orientations are prescribed and fixed and the magnetization and the magnetic field are assumed to be given as well. By means of a formal two-scale expansion, effective equations are obtained in the homogenization limit.

In \cite{NikaVernescu20, DangGorbBolanos21generalized, DangGorbBolanos21SIMA, DangGorbBolanos23}, the authors consider models of period suspensions where (Navier)-Stokes equations are coupled to the Maxwell equations. In the homogenization limit they rigorously obtain effective coupled systems. As in the works mentioned before, the problem is static. The magnetization varies inside of the individual particles. This and the non-dilute setting makes the considered problem more relevant for magnetorheological fluids, where the particles are larger than in ferrofluids.

Even suspensions of force- and torque-free particles are well known to give rise to an effective viscosity of the fluid, which is of the order of the particle volume fraction $\phi$ for small $\phi$. We do not give an exhaustive discussion of the rigorous results on this topic but refer the reader to \cite{DuerinckxGloria23} and the references therein.
Most of these results concern static problems where the particle positions (and orientations in the case of nonspherical particles) are prescribed.
The particle evolution of inertialess suspensions has been investigated in \cite{Hofer&Schubert, Hofer18MeanField, Mecherbet19, HoeferLeocataMecherbet22, HoferMecherbetSchubert22, Gerard-VaretHoefer24, Duerinckx23}.
In \cite{Hofer18MeanField, Mecherbet19}, the transport-Stokes equation has been derived for spherical buoyant  particles. The increase of viscosity, which is a higher order effect in this setting, has been investigated in \cite{Hofer&Schubert}. In \cite{HoeferLeocataMecherbet22}, the authors considered a simplified model of a Brownian nonspherical suspension and derived the effective elastic stress on the fluid. In \cite{HoferMecherbetSchubert22, Gerard-VaretHoefer24, Duerinckx23} the derivation of effective fully coupled multiscale models for dilute non-spherical non-Brownian suspensions is studied. 
In particular, it has been pointed out that the \enquote{naive} mean-field limit does not capture the evolution of the particle orientations to order $O(\phi)$.

Compared to the aforementioned result, the main novelty of the present work arises from the spin flip of the particles. In particular, the main difficulty is to make rigorous that both local time averages of the spin of individual particles as well as  local phase space averages of  spins of particles are well approximated by the local equilibrium law $\tanh \big( b \, H(t,x) \cdot \zeta \big)$. We comment on the strategy of the proof after the statement of Theorem \ref{main theorem}.

\subsection{Modeling of the microscopic problem}
\label{subsection Microscopic problem}

In the following, we introduce the microscopic model that serves as the starting point of the analysis. For more details on the modeling, in particular also the simplification  we make,  the chosen nondimensionalization and a short discussion of typical orders of magnitudes, we refer to Appendix \ref{sec:app}.

\paragraph*{Particle and fluid domain}

Let $\mathcal B \subset B_1(0)$ be a fixed reference particle, a smooth compact and connected set with rotational symmetry, i.e.
\begin{align}\label{rod_reference}
	R \B = \B \quad \text{for  all } R \in SO(3) \text{ with } R e_3 =e_3,
\end{align}
where $e_3 = (0,0,1) \in \R^3$.

We consider $N \in \N$ particles with centers $X_i \in \R^3$ orientations $\xi_i \in \S^2$, $1 \leq i \leq N$.
We choose a rotation matrix $R{\xi_i} \in SO(3)$  such that $R_{\xi_i} e_3 = \xi_i$.
For a scaling parameter $r>0$, we then denote the $i$-th particle by
\begin{align}
    \B_i := X_i + r R_{\xi_i} \B.
\end{align}

For a given vector $\sigma \in \spinstate^N$, the $i$-th element $\sigma_i = \pm 1$ corresponds to the spin of the $i$-th particle. 

\paragraph*{Fluid equation}
Given an external magnetic field $H(t,x)$, and $(X(t),\sigma(t),\xi(t)) \in \R^{3N} \times (\S^2)^N \times \spinstate^N$, the fluid velocity $u_N$ is given as the unique solution in $u_N(t,\cdot) \in \dot H^1(\R^3)$ to  the Stokes equations with balance of forces and torques
\begin{align} \label{u_N}
\left\{\begin{array}{rl}
    -   \Delta u_N + \nabla p_N = 0,  \quad \dv u_N = 0 &\quad \text{in }  \R^3\setminus \bigcup_i \B_i, \\
     D u_N = 0 &\quad \text{in }  \bigcup_i \B_i, \\
        -\int_{\partial \B_i} \Sigma[u_N,p_N^{\sigma, \xi}] n \dd x  =  \sigma_i  r^3 |\B|   (\nabla  H)^T \xi_i  &\quad \forall i, \\
    -\int_{\partial \B_i} (x-X_i)   \times \Sigma[u_N,p_N^{\sigma, \xi}] n  \dd x  =  \sigma_{i} r^3 |\B|    \xi_i  \times  H(t,X_i)  &\quad \forall i.
    \end{array} \right. 
\end{align} 
Here, $\dot H^1(\R^3) = \{ v \in L^{6}(\R^3) : \nabla v \in L^2(\R^3)\}$  is a homogeneous Sobolev space, $\Sigma[u,p] = 2 D u - p \Id $, where $D u = \tfrac 1 2 (\nabla u  + (\nabla u)^T)$ is the symmetric gradient, and $n$ denotes the outer unit normal. 

\paragraph{Particle evolution}
We artificially fix the particle positions $X_i$, $1\leq i \leq N$, and study the evolution of the particle orientations and spins coupled to the fluid flow.
The evolution combines a deterministic motion of the orientations and random jumps of the spins. For simplicity, we consider deterministic initial data $(\xi^0,\sigma^0) \in (\S^2)^N \times \spinstate^N$. All results can be adapted to random initial data.

The spins of the particles flip according to the following time continuous Markov chain. For each particle $\B_i$, the probability for the spin $\sigma_i(t)$  to switch between $t$ and $t + \mathrm{d}t$ is given by 
\begin{align}
\label{def:intensities lambda}
\begin{aligned}
    \frac{1}{\e} \lambda_i(t, X,\xi(t), \sigma(t)) \dd t &= \frac 1 \eps \lambda^{\sigma_i(t)}(t, X_i,\xi_i(t)) \dd t, \\
    \lambda^{\pm}(t,x,\zeta) &=\exp\left(\mp  b H(t,x) \cdot \zeta \right),
    \end{aligned}
\end{align}
where $\eps > 0$ and $b>0$ are two constants related to physical parameters. We denote $t^{(0)} = 0$ and $t^{(k)}$ the time of the $k-th$ jump (of the spin of any of the $N$ particles). 
Between successive jumps, $t \in [t^{(k)}, t^{(k+1)}]$, the evolution of the orientations $(\xi_i)_i$ is given through the ODEs
       \begin{align} \label{ODE.Xi}
        \left\{ \begin{array}{l}
              \dot \xi_i(t) =  \frac 1 2 \curl u_N(t,X_i) \times  \xi_i(t)  = \nabla u_N(t, X_i)   \xi_i=: V(t, \xi(t),\sigma(t)) , \\
              \xi_i(t^{(k)}) = \xi_i^0.
              \end{array}\right.
\end{align}
Due to the quasi-stationarity of the fluid equation, $V$ only depends on time through the time dependence of the magnetic field $H$. Note that $V$ also implicitly depends on the fixed particle positions.
Let $\Phi \colon \R \times \R \times (\S^2)^N \times \{\pm1\}^N \to (\S^2)^N \times \{\pm1\}^N$ be the associated flow, i.e.
\begin{align}
    \partial_s \Phi(s;t,\bar \xi, \bar \sigma) =  (V(s,\Phi(s;t,\bar \xi, \bar \sigma)),0), \qquad \Phi(t;t,\bar \xi,\bar \sigma) = (\bar \xi, \bar \sigma).
\end{align}
Note that this flow is well-defined since $V$ is Lipschitz continuous with respect to $\bar \xi$.

More precisely, one can construct the stochastic process $Z(t) = (\xi(t),\sigma(t))$ as follows. It suffices to construct the jump times $t^{(k)} \in [0,\infty)$ and the states at the jump times  $Y^{(k)} = (\xi^{(k)},\sigma^{(k)}) \in (\S^2)^N \times \{\pm1\}^N$. Then, the stochastic process is defined via
\begin{align}
    Z(t) = \Phi(t;t^{(k)},Y^{(k)}) \quad \text{for } t \in [t^{(k)},t^{(k+1)}). 
\end{align}
The jump times and states are inductively defined as follows, starting from the deterministic  initial configuration $(\xi^0,\sigma^0) \in (\S^2)^N \times \spinstate^N$.
\begin{itemize}
    \item  Set $t^{(0)} = 0$, $Y^0 = (\xi^0,\sigma^0)$, sample $E^{(k)}_{j} \sim \mathrm{Exp}(1)$, $k \in \N, 1 \leq j \leq N$, independently.
    \item Assume $t^{(l)}, Y^{(l)}$, $1 \leq l \leq k$ are constructed.  Then define
    \begin{align}
        t^{(k+1)} := \inf \left\{ t > t^{(k)} : \exists \, 1 \leq j \leq N   \text{ s.t. } \frac 1 \eps \int_{t^{(k)}}^t \lambda_j(s,\Phi(s;t^{(k)},Y^{(k)})) \dd s \geq  E^{(k+1)}_{j} \right\}
    \end{align}
    and let $1 \leq j_{flip} \leq N$ be  minimal such that $\int_{t^{(k)}}^{t^{(k+1)}} \lambda_{j_{flip}}(s,\Phi(s;t^{(k)},Y^{(k)})) \dd s \geq  E^{(k+1)}_{j_{flip}}$.
    Set 
    \begin{align}
        Y^{(k+1)} = \tau_{j_{flip}} \Phi(t^{(k+1)};t^{(k)},Y^{(k)}),
    \end{align}
    where $\tau_j (\bar \xi, \bar \sigma) =  (\bar \xi,\tau_j \bar \sigma)$ is the flip operator, defined through 
    \begin{align}
        (\tau_j \bar \sigma)_i = (-1)^{\delta_{ij}} \bar \sigma_i.
    \end{align}
\end{itemize}

\paragraph{PDE formulation}

The process $Z(t)$ defined like this is an (inhomogeneous) continuous-time Markov process, a so called piecewise deterministic Markov process.  We refer to \cite{Davis} for an introduction into such processes. 
In particular, the process $Z(t)$ is characterized in terms of its generator that leads to the following PDE formulation.
For $N>0$,  $T>0$, 
    let $p_{\e} =  p_\e(t,\bar \xi ,\bar \sigma) \in C^0_w([0,T], \mathcal{P}({(\S^2)^N} \times \{\pm 1\}^N))$ be the unique distributional solution of the system
    \begin{align} \label{generator micro system}
    \left\{\begin{array}{rl}
        \partial_t p_\e + \dv_{\bar \xi} (V p_\e) = \frac 1 \eps \sum_{i=1}^N \left(\tau_i(\lambda_i   p_\e) -  \lambda_i  p_\e \right) & \quad \text{in } (0,\infty) \times ( \S^2)^N \times \{\pm 1\}^N,\\
        p_\eps(0,\cdot)  = \delta_{(\xi_i^0, \sigma^0)} & \quad \text{in }  ( \S^2)^N \times \{\pm 1\}^N.
        \end{array} \right.
    \end{align}
Here, $C^0_w([0,T], \mathcal{P}({(\S^2)^N} \times \{\pm 1\}^N))$ denotes the space of weakly continuous functions $t \mapsto \mathcal{P}({(\S^2)^N} \times \{\pm 1\}^N)$. Moreover, the flip operator acts on functions via 
\begin{align} \label{tau_i.operator}
    \tau_i f(\cdot,\bar \sigma) = f(\cdot,\tau_i \bar \sigma).
\end{align}

Then, $p_\eps(t,\cdot)$ is the probability distribution of the piecewise deterministic Markov process process $Z_t$ that starts from the same initial distribution. More precisely, for all $\varphi \in C({(\S^2)^N} \times \{\pm 1\}^N)$
\begin{align}
    \E[\varphi(Z(t))] = \int_{{(\S^2)^N} \times \{\pm 1\}^N} \varphi \dd p_\e(t,\cdot).
\end{align}
In the following, we will only rely on this PDE formulation of the stochastic process.

\subsection{Statement of the main result} \label{sec:mainResult}

\paragraph{Assumptions on the particle positions}
We assume that the particle centers $X \in \R^{3N}$ are contained a compact set, uniformly in $N$, i.e., 
\begin{align} \label{K} \tag{H1}
    \exists K \subset \R^3 \text{ compact } ~ \forall 1 \leq i \leq N \quad  X_i \in K.
\end{align}

We enforce that the particles do not overlap by assuming
\begin{align}
    \label{ass:non.overlap} \tag{H2}
    \exists \, \theta > 2 ~  \quad \underset{i \neq j}{\min} |X_i - X_j| \geq 2 \theta r.
\end{align}

Moreover, we assume the following strong separation condition
\begin{equation}
    \label{assumption separation} \tag{H3}
    \exists \, c_0 > 0  \quad \underset{i \neq j}{\min} |X_i - X_j| \geq c_0 N^{- \frac{1}{3}}.
\end{equation}

We remark that \eqref{K} and \eqref{ass:non.overlap} imply an upper bound for $r$ (for $N \geq 2$). 
Moreover, if the \enquote{volume fraction} $\phi := Nr^3$ is sufficiently small, \eqref{assumption separation} implies \eqref{ass:non.overlap}.

\paragraph{Effective equations}
Our goal is to characterize the effective behavior of the fluid flow as well as the particle dynamics. The latter we will capture through the particle density. More precisely, given initial particle orientations and spins  $(\xi^0, \sigma^0) \in (\S^2)^N \times \spinstate^N$, we consider their evolution by the piecewise deterministic Markov process above.
We denote by $h_N(t,\cdot) \in \mP(\R^3 \times \S^2 \times \spinstate)$ and $f_N(t,\cdot) \in \mP(\R^3 \times \S^2 ) $ the empirical densities
\begin{align}
    h_N(t, \d x, \d\zeta, \d \varsigma,) &= \frac 1 N \sum_{i = 1}^N \delta_{(X_i,\xi_i,\sigma_i)}, \\
     f_N(t,  \d x, \d\zeta) &= \frac 1 N \sum_{i = 1}^N \delta_{(X_i,\xi_i)}.
\end{align}

Our main result asserts that $(h_N,u_N)$ is well approximated by the following limiting behavior.
The fast spin flip $\eps^{-1}$ drives $h_N$ towards a (local) equilibrium configuration where fraction of particles in spin states $\pm 1$ is given by
\begin{align}
        m^{\pm}(t,x,\zeta) &= \frac{\lambda^{\mp}(t,x,\zeta)}{\lambda^+(t,x,\zeta) + \lambda^-(t,x,\zeta)} \label{m_pm}.
\end{align}
Hence, effectively, we obtain
\begin{align} \label{h.eff}
    h &= f \otimes (m^+ \delta_{+1} + m^- \delta_{-1}),
\end{align}
where the slightly abusive notation $f \otimes (m_+ \delta_{+1} + m_- \delta_{-1})$ means 
\begin{align}
    \int \varphi(x,\zeta,\varsigma) \dd (f \otimes  m_{\pm} \delta_{\pm 1}) =
    \int \varphi(x,\zeta,\varsigma)  m_\pm(t,x,\zeta) f(t,\d x,\d\zeta) \delta_{\pm1} (\d\varsigma).
\end{align}
We denote the average spin in the local equilibrium by
\begin{align} \label{m^0}
    m^0(t,x,\zeta) = (m^{+}- m^-)(t,x,\zeta) = \tanh \big( b \, H(t,x) \cdot \zeta \big).
\end{align}
Then, the effective equation for $f$ is characterized by a decoupling to leading order in the particle volume fraction: the particles evolve as if they were isolated from the others,\begin{align}\label{eq:f}
\left\{\begin{array}{rl}
    \partial_t f + \dv_\zeta ( m_0 \gamma P_{\zeta^\perp} H  f) = 0 \quad  &\text{in } (0,\infty) \times \R^3 \times \S^2,\\
    f(0,\cdot) = f_0\quad  &\text{in } \R^3 \times \S^2.
\end{array} \right.
\end{align}
Here, the parameter $\gamma > 0$ is the rotational resistance of the reference particle. It is obtained by solving an exterior Stokes problem outside of the reference particle $\B$, see Section \eqref{sec:single} below. 
The initial datum $f_0 \in \mathcal{P}(\R^3 \times \S^{2})$, one should  think of as the limit of $f_N^0$ as $N \to \infty$.

Finally, the fluid velocity $u_N$ is well approximated by
the unique weak solution to 
\begin{align} \label{u.eff}
\left\{\begin{array}{rl}
    - \Delta u + \nabla p =  \phi \int_{\S^2}   m_0 (\nabla H)^T \zeta f \dd \zeta + \phi \dv \int_{\S^2} m_0 ((\Id + \mR) (\zeta \wedge H) f \dd \zeta, \quad  &\text{in } (0,\infty) \times \R^3, \\
    \dv u = 0 \quad  &\text{in } (0,\infty) \times \R^3.    
\end{array} \right. \quad
\end{align}
Here $a \wedge b$ denotes the skewsymmetric part of the tensor product, i.e. $a \wedge b = \tfrac 1 2 (a \otimes b - b \otimes a) $, the tensor $\mR = \mR(\zeta) \in \mathcal L(\skew(3),\sym_0(3))$ is defined in \eqref{mR} below and 
\begin{align} \label{phi}
    \phi = N r^3 |\B|.
\end{align}

The precise formulation of our main result is the following.
\begin{theorem}
\label{main theorem}
Let $H \in W^{2,\infty}((0,T) \times \R^3))$. For parameters $N \in \N$, $r> 0$, $0 < \eps < 1$, let $(X, \xi^0, \sigma^0) \in (\R^3 \times \S^2 \times \spinstate)^N$ be initial particle configurations satisfying \eqref{K}--\eqref{assumption separation}.
Let $u_N(t)$, $(\xi_i(t))_{1 \leq i \leq N}$ and $(\sigma_i(t))_{ 1 \leq i \leq N}$ the  stochastic process defined in Section \ref{subsection Microscopic problem} with particle centers.

Let $f_0 \in \mathcal{P}(\R^3 \times \S^{2})$, and let $f, h, u $
be defined through \eqref{eq:f}, \eqref{h.eff} and \eqref{u.eff}.

Then, for all $T>0$ there exists $C >0$ depending only on $T$, the reference particle $\B$, the set $K$ from \eqref{K}, the constants $c_0$ and $\theta$ from \eqref{ass:non.overlap}, \eqref{assumption separation} and on $\|H\|_{W^{2,\infty}((0,T) \times \R^3))}$ such that forall $t \in [0,T]$, all $p \in (1,3/2)$ and all $x \in \R^3$
\begin{align}\label{f_N,f}
    \mathbb{E} \big[W_2^2(f_N(t), f(t))\big] &\leq  \Big(C\big(\eps  + \phi^2 \big) t + W_2^2(f_N^0, f^0) \Big) e^{C t}, \\
    \mathbb{E} \big[W_2^2( h_N(t), h(t) \big] &\leq C \left( N^{-\frac 2 9} +  \eps^{\frac 1 4}  + \phi ^{\frac 1 2}  + W_2(f_N^0, f^0) + e^{\frac{-t}{C\eps}} \right)  \label{h_N,h},\\
\label{u_N.u}
  \frac 1  \phi \|u_N(t,\cdot) - u(t,\cdot)\|_{L^p(B_1(x))} &\leq C \left(\phi^{\frac 1 2} + r^{\frac 3 p - 2} + W^{\frac 3 p - 2}_2(h_N,h) \right) .
\end{align}
\end{theorem}

\begin{rem}
    \begin{enumerate}
        \item In \eqref{f_N,f}--\eqref{u_N.u}, $W_2$ denotes the $2$-Wasserstein distance.
        \item Clearly, norms of $u$ given through \eqref{u.eff} are proportional to  $\phi$. We consequently put the factor $\frac 1 \phi$ on the right-hand side of \eqref{u_N.u} to emphasize that the estimate is nontrivial in the sense that it gives a higher order approximation with respect to $\phi$.
    \end{enumerate}
\end{rem}

\subsection{The parameter \texorpdfstring{$\gamma$}{gamma} and the tensors  \texorpdfstring{$\mR(\xi)$}{R}} \label{sec:single}

We provide here the definitions of the rotational resistance  parameter $\gamma$ and the  tensor  $\mR$ that relates the torque on a particle to the moment of stress. These parameters appear in the limit system \eqref{eq:f}, \eqref{u.eff}.

\paragraph{Definition of rotational resistance  parameter $\gamma$}
Let $T,F \in \R^3$ be a given torque and force. Recall that for $\zeta \in \S^2$, we denote by $R_{\zeta} \in SO(3)$ a rotation matrix  such that $R_{\zeta} e_3 = \zeta$. 
Then, we consider the solution $v_{T,F} \in \dot H^1(\R^3)$ to the single particle problem 
\begin{equation} \label{v_T,F} \left\{
\begin{array}{ll}	-  \Delta v_{T,F} + \nabla p_{T,F} = 0, \quad \dv v_{T,F} = 0 &\quad  \text{in } \R^3 \setminus (r R_\zeta\B), \\
	D v_{T,F} = 0  & \quad \text{in } r R_\zeta\B, \\
    - \int_{\partial (r R_\zeta \B)} \Sigma[v_{T,F},p_{T,F}] n \dd x = F, \\
    - \int_{\partial (r R_\zeta\B)} x \times \Sigma[v_{T,F},p_{T,F}] n \dd x = T .
	\end{array} \right.
\end{equation}
Through linearity, scaling and rotational symmetry it can be shown that  such that the angular velocity is independent of $F$ and that there exists  parameters $\gamma, \tilde \gamma > 0$ that depend only on $\B$ such that
\begin{align} \label{angular.single.1}
    \frac 1 2 \curl v_{T,F}(0)  = \frac{1}{r^3 |\B|} \left( \gamma P_{\zeta^\perp} T + \bar \gamma (T \cdot \zeta) \zeta \right). 
\end{align}
In particular, for $T = r^3 |\B| \zeta \times \bar H$, we have
\begin{align} \label{angular.single.2}
    \frac 1 2 \curl v_{T,F}(0) \times \zeta =   \nabla v_{T,F}(0) \zeta = \fint_{r R_\zeta\B}  \nabla v_{T,F} \zeta \dd x = \gamma (\zeta \times \bar H) \times \zeta  = \gamma P_{\zeta^\perp} \bar H.
\end{align}

For details, we refer to \cite[Chapter 5]{KimKarilla13}.






\paragraph{Definition of the tensor $\mR$}

The map $\mR: \S^2 \to \mathcal L (\skew(3),\sym_0(3))$ is defined as follows.
Given $A \in \skew(3)$, consider the Stokes problem
\begin{equation}\left\{
\begin{array}{ll}	
-  \Delta v_A + \nabla p_A = 0, \quad \dv v_A = 0 &\quad  \text{in } \R^3 \setminus (R_\zeta\B), \\
	D v_A = 0  & \quad \text{in } R_\zeta\B, \\
    \int_{\partial (R_\zeta \B)} \sigma(v_A,p_A) n \dd x  = 0, \\
    \int_{\partial (R_\zeta\B)} x \wedge \sigma(v_A,p_A) n  \dd x= A, \\
    v_A(x) \to 0 & \quad \text{as } |x| \to \infty.
	\end{array} \right.
\end{equation}
Then, we define
\begin{align} \label{mR}
    \mR(\zeta) A := \int_{\partial (R_\zeta \B)} x \otimes_s^0  \sigma(v_A,p_A) n \dd x,
\end{align}
where $a \otimes_s^0 b$ denotes the traceless symmetric tensor product, i.e. $a \otimes_s^0 b = \tfrac 1 2 (a \otimes b + b \otimes a) - \tfrac 1 3 a \cdot b \Id $.

\subsection{Discussion of the main result}
The effective multiscale model \eqref{eq:f} asserts that (to leading order in $\phi$) the -- \emph{time-averaged} -- angular velocity of a particle with orientation $\xi_i$ is given by $m_0(t,X_i,\xi_i) \gamma P_{\xi_i^\perp} H(t,X_i)$. The parameter $\gamma$ is defined in such a way that $ \gamma P_{\zeta^\perp} H$  is the angular velocity of a particle with orientation $\zeta$ on which a torque $\xi_i \times H |\mB_i|$ acts (see \eqref{angular.single.2}). We emphasize that the microscopic angular velocity is pointwise in time well approximated by $\sigma_i \gamma P_{\xi_i^\perp} H$ (cf. Proposition \ref{prop:angular.velocity}) and not by $m_0(t,X_i,\xi_i) \gamma P_{\xi_i^\perp} H(t,X_i)$. However, for the evolution of the orientation, only local \emph{time averages} of the angular velocity are  relevant which makes $m_0$ appear for $\eps \ll 1$.

On the other hand, the fluid velocity is (in the absence of external source terms on the fluid itself) of order $\phi$. The two right-hand side terms in \eqref{u.eff} are the effective Kelvin force and stress, respectively. We emphasize that $u_N$ is close to $u$ \emph{pointwise} in time. Consequently, in contrast to the appearance of $m_0$ in \eqref{eq:f}, the average spin $m_0$ in these right-hand side terms has to be understood as \emph{space-averaged} spins   of the particles.

For spherical particles, the map $\mR$ defined in \eqref{mR} satisfies $\mathcal R = 0$. In this case, and if $\curl H = 0$ (which is typically assumed inside of the ferrofluid because there are no currents)  we can rewrite the momentum equation of the Stokes equation as
\begin{align} \label{force.stress.classical.form}
    - \Delta u + \nabla p = M \cdot \nabla H + \dv (M \wedge H),
\end{align}
where $M$ is the  magnetization
\begin{align} \label{M.magnetization}
    M(t,x) = \phi \int_{\S^2} m_0(t,x,\zeta) f(t,x, \d \zeta).
\end{align}
The right-hand side of \eqref{force.stress.classical.form} coincides with the classical form of the magnetic source terms in macroscopic ferrofluids models with \enquote{internal rotation} (that is $M$ is not a priori assumed to be parallel to $H$, see e.g. \cite{Shliomis02, Rosensweig02}) where the evolution of $M$ is given through a PDE instead of \eqref{M.magnetization}.
We emphasize that for non-isotropic particles $\mR = \mR(\zeta) \neq 0$. In this case the last right-hand side term in \eqref{u.eff} cannot be expressed solely in terms of the the magnetization $M$. This is clearly a shortcoming of the macroscopic ferrofluid models. Moreover, we do not obtain a closed PDE for $M$ out of the PDE for $f$. This means, that macroscopic ferrofluid models, can only been obtained from the multiscale formulation via approximate closures or suitable singular limits.

\paragraph{Possible generalizations}

One might wonder whether it is possible to include in \eqref{eq:f} a term of the form $\dv_\zeta ((\mathcal M \nabla u) \zeta f)  $ --  for a suitable matrix $\mathcal M$ that depends on the reference particle -- to obtain an approximation which yields an error $o(\phi)$ for $W_2(f_N,f)$. This is not the case as shown in \cite{HoferMecherbetSchubert22, Gerard-VaretHoefer24}: The interaction of the particles regarding the angular velocities has a singularity  like $|X_i - X_j|^{-3}$, which is too singular for the naive mean-field limit to hold. 
In \cite{Gerard-VaretHoefer24}, a correction term in addition to $\dv_\zeta ((\mathcal M \nabla u) \zeta f)  $ has been included in \eqref{eq:f} which yields an accuracy $o(\phi)$ but only in very weak norms.

As in \cite{Gerard-VaretHoefer24, Duerinckx23}, we could add a given source term $h$ on the right-hand side of the Stokes equation \eqref{u_N} and allow for particle translations (for $ N r \gg N^{-1/3}$ and $\phi \log N \ll 1$ which prevents particle clustering by ensuring that particle translational velocities of close particles do not differ too much). 
Then, instead of \eqref{eq:f}--\eqref{u.eff}, the effective system would read
\begin{align*}
\left\{\begin{array}{l}
    \partial_t f + \dv_x (u f) +   \dv_\zeta (((\mathcal M \nabla u) \zeta   + m_0 \gamma P_{\zeta^\perp} H ) f) = 0 \\
    -  \dv ((2 + \phi \mathcal V[f]) Du ) + \nabla p =  h +  \phi \int_{\S^2}   m_0  (\nabla H)^T \zeta  f \dd \zeta + \phi \dv \int_{\S^2} m_0 (\Id + \mR) (\zeta \wedge H) f \dd \zeta,\\
    \dv u = 0,
    \end{array} \right.
\end{align*}
where $\mathcal V [f] = \int_{\S^2} S(\xi) h \dd \xi$ accounts for the increased viscosity in terms of a linear map $S(\xi) \in \mathcal L(\sym_0(3),\sym_0(3))$ that depends on the reference particle $\B$. We chose not to include $h$ and particle translations for the sake of the simplicity of the presentation.

It would also be interesting to include rotational Brownian motion into the system. Then, one would expect a Fokker-Planck equation for $f$. In view of \cite{HoeferLeocataMecherbet22}, one also expects an additional elastic stress in the fluid equation.

\subsection{Elements of the proof and outline of the rest of the paper}

The rest of the paper is devoted to the proof of Theorem \ref{main theorem}. In Section \ref{sec:fluid.approx}, we show suitable approximations for the fluid velocity $u_N$ that is the solution to the static problem \eqref{u_N} for a given $(\bar \xi , \bar \sigma) \in (\S^2)^N \times \spinstate^N$. 
These are obtained 
through the superposition of single particle solutions of the Stokes equations in the dilute regime $\phi \ll 1$. In particular, we show in Proposition \ref{prop:angular.velocity} that, for $\phi \ll 1$ and $r \ll 1$, the angular velocity $\nabla u_N \xi$ is close to $\bar \sigma_i \gamma P_{\bar \xi_i}^\perp H(t,X_i)$.  
The methods for these estimates are relatively standard, but the proofs contain some new aspects. In particular, due to the critical singularity of the interaction and since we do not assume $\phi \to 0$ as $\N \to \infty$ with a certain rate, we cannot use the method of reflections in the same way as in \cite{Hofer18InertialessLimit, Gerard-VaretHoefer24, HoeferLeocataMecherbet22} to get $L^\infty$ control on $\nabla u_N$. Therefore, we do not obtain estimates for the aforementioned angular velocities that are uniform over all particles. Instead, using Calderon-Zygmund estimates adapted from \cite{HillairetWu19, Gerard-VaretHoefer21}, we only get an estimate on these angular velocities in $l_2$ with respect to the particles. In contrast to \cite{Hofer18InertialessLimit, Gerard-VaretHoefer24, HoeferLeocataMecherbet22}, this turns out to be sufficient in our case, because we neglect particle translations. 

In Section \ref{sec:relaxation}, we show that, in the regime of fast spin jump rate $\eps \ll 1$,  the piecewise deterministic Markov jump process $(\xi,\sigma)$ can be approximated through deterministic decoupled ODEs for the orientation $\xi^\ast$ that involves   $m^0$. Moreover,  we prove  $m^0$ to be close to the expectation of the spins after an initial layer of order $\eps$. We also show that the correlations of the spins are small.  The proof of these approximations is based on a rather delicate relative energy argument for $\mathcal W = \E[|\xi - \xi^\ast|^2]$:  since the relaxation mechanism occurs in the spins only, the first time derivative of $\mathcal W$ does not feature any dissipation. However, the leading order term of this time derivative can be controlled because its own time derivative is dissipative.

Section \ref{sec:proof.main} is dedicated to the conclusion of the proof of Theorem \ref{main theorem}. Estimate \eqref{f_N,f}  follows directly from the results of Section   \ref{sec:relaxation} and a stability estimate for \eqref{eq:f}. Moreover, estimate \eqref{u_N.u} is a consequence of \eqref{h_N,h} and the results of Section \ref{sec:fluid.approx}.
The main remaining effort thus goes into the proof of \eqref{h_N,h}. In order to estimate this Wasserstein distance, we discretize the phase space $\R^3 \times (\S^2)^N $ into small cubes of length $\delta$. 
The key point is then to show through an appropriate law of large numbers that the total spin of the particles inside each of those cubes is close to the  total amount of spin predicted by the effective density $h$. For the proof we need the correlation estimate on the individual particle spins shown in Section \ref{sec:relaxation}.

\subsection{Notation}

In the following we write  $C, c$ for any constant that depends only on the quantities specified in Theorem \ref{main theorem}. Moreover, we write  $A \lesssim B$ whenever there exists $C > 0$ such that $A \leq C B$ where $C$ is a constant that depends on the quantities specified in Theorem \ref{main theorem}. The notation $\gtrsim$ is used analogously.

Many quantities like $\xi \colon (0,\infty) \to  (\S^2)^N, \sigma \colon (0,\infty) \to \spinstate^N$ and $u_N \colon (0,\infty) \to \dot H^1$, the solutions to the microscopic process defined in Section \ref{subsection Microscopic problem},  implicitly depend on the parameters $N,\eps,r$. For the sake of a leaner notation, we only make explicit some of these dependencies. 

For $p \in [1,\infty]$, we write $\|\cdot\|_p$ short for $\|\cdot \|_{L^p(\R^3)}$.

\section{Approximations for the fluid velocity} \label{sec:fluid.approx}

In this section, we consider centers $X \in \R^{3N}$ satisfying \eqref{K}--\eqref{assumption separation} and fixed orientations and spins  $(\bar \xi, \bar \sigma) \in (\S^2)^N \times \spinstate^N$ and a scaling parameter $r>0$. 
Let $ \bar u_N  \in \dot H^1(\R^3)$ be the solution to \eqref{u_N} with $H = \bar H \in W^{2,\infty}(\R^3)$ with $\|\bar H\|_{W^{2,\infty}(\R^3)} \lesssim 1$.

For  $j \in \{1, \dots, N \}$, we introduce $v_j \in \dot{H}^1(\R^3)$ the solution to the single particle problem 
\begin{align} \label{v_j}
\left\{\begin{array}{ll}
    -   \Delta  v_j + \nabla  p_j = 0,  \quad \dv  v_j = 0 \quad \text{in }  \R^3 \setminus \B_j, \\
     D  \, v_j = 0 \quad \text{in }  \B_j, \\
      - \int_{\partial \B_j} \Sigma[v_j,p_j] n \dd x  =  \bar \sigma_j r^3 |\B| (\nabla  \bar H(X_j))^T \bar \xi_j  , \\
      - \int_{\partial \B_j} (x-X_j)   \times \Sigma[v_j,p_j] n  \dd x   =    \bar \sigma_j r^3 |\B| \bar \xi_j \times  \bar H(X_j)=0.
    \end{array} \right.
 \end{align}  
 Note that $v_j$ implicitly depends on $X_j, \bar \xi_j, \bar \sigma_j$.

 Moreover, we introduce another approximation $\tilde v_j \in L^p_{loc}(\R^3)$, $p \in [1,3/2)$ the solution to 
\begin{align} \label{tilde.v_j}
    -\Delta \tilde v_j + \nabla \tilde p_j = F_j \delta_{X_j} + \dv\left( S_j \delta_{X_j}\right) \quad \dv \tilde v_j = 0 \qquad \text{in } \R^3
\end{align}
where
\begin{align}
    F_j = |\B_j| \bar \sigma_j (\nabla  \cdot \bar H(X_j))^T \bar \xi_j , \label{F_j} \\
    T_j = |\B_j| \bar \sigma_j \bar \xi_j \times \bar H(X_j)), \\
    S_j = |\B_j| \bar \sigma_j  (\Id + \mR(\xi_j))(\bar \xi_j \wedge \bar H(X_j)) .\label{S_j}
\end{align}

\begin{lem}\label{lem:v_j.tilde.v_j}
    For all $x \in \R^3$ and $p \in (1,3/2)$,
    \begin{align} \label{v_j.tilde.v_j}
        \|v_j - \tilde v_j\|_{L^p(B_1(x))} \lesssim r^{1+3/{p}}
    \end{align}
\end{lem}
\begin{proof}
    Let $g \in L^{p'}(\R^3)$, with $\supp g \subset \overline{B_1(x)}$. Let $\varphi \in \dot H^1(\R^3)$ be the solution to 
\begin{align}
    -   \Delta  \varphi + \nabla  p = g,  \quad \dv  \varphi  = 0 \quad \text{in }  \R^3.
 \end{align}
Then,
 \begin{align} \label{g.v_j.tilde}
    \begin{aligned}
     \int_{\R^3} g \cdot (v_j -  \tilde v_j) &= \int_{\R^3} 2 D \varphi  : D ( v_j -  \tilde v_j) \\
     &= \int_{\partial \B_j} \Sigma[v_j,p_j] n \cdot \varphi \dd x -  F_j \cdot \varphi(X_j)(0) + S_j : \nabla \varphi(X_j) \\
     &=\int_{\partial \B_j} \Sigma[v_j,p_j] n \cdot (\varphi(x) - \varphi(X_j) - (x-X_j) \cdot \nabla \varphi(X_j)),
     \end{aligned}
 \end{align}
where we used the definition of $v_j$, $F_j$ and $S_j$ from \eqref{v_j} and \eqref{F_j}--\eqref{S_j} as well as the definition of $\mR$ from \eqref{mR} 
and $\dv \varphi = 0$.
 
 Let $\eta \in C_c^\infty(B_2(X_j))$ be a cut-off function  with $\|\nabla^l \eta\|_\infty \lesssim r^{-l}$, $l = 0,1$ and define
 \begin{align}
     \tilde \psi(x) &:= \eta (x) \left(\varphi(x) - \varphi(X_j) - (x-X_j) \cdot \nabla \varphi(X_j)\right), \\
     \psi &= \psi - \Bog(\dv \tilde \psi),
 \end{align}
 where $\Bog: L^2_0(B_{2r}(X_j) \setminus \B_j) \to H^1_0(B_{2r}(X_j) \setminus \B_j)$ is a Bogovski operator, i.e. an operator that satisfies
 \begin{align}
     \dv \Bog(h) = h. 
 \end{align}
 Here $L^2_0(U) = \{h \in L^2(U) : \int_U h = 0\}$.
It is classical that such an operator exists, and by scaling considerations we can choose it such that
\begin{align}
    \|\nabla \Bog(h)\|_{L^2(U)} \lesssim \|h\|_{L^2(U)}.
\end{align}
Hence, 
 \begin{align}
     \|\nabla \psi\|_{L^2(\R^3)} &\lesssim r^{-1} \|\varphi(x) - \varphi(X_j) - (x-X_j) \cdot \nabla \varphi(X_j)\|_{L^2(B_{2r(X_j))}} \\
     &+ \|\nabla (\varphi(x) - \varphi(X_j) - (x-X_j) \cdot \nabla \varphi(X_j))\|_{L^2(B_{2r(X_j))}} \\
     &\lesssim r^{3/2 + \alpha} [\nabla \varphi]_{\alpha}.
 \end{align}
 Hence, 
 \begin{align} \label{phi.to.psi}
    \begin{aligned}
     \left|\int_{\partial \B_j} \Sigma[v_j,p_j] n \cdot (\varphi(x) - \varphi(X_j) - (x-X_j) \cdot \nabla \varphi(X_j))\right| &= \left|\int_{\partial \B_j} \Sigma[v_j,p_j] n \cdot \psi \right| \\
     &= \left|\int_{\partial \B_j} D v_j \cdot D \psi \right| \\
     &\lesssim \|D v_j\|_2  r^{3/2 + \alpha} [\nabla \varphi]_{\alpha}.
     \end{aligned}
 \end{align}
 By regularity theory for the Stokes equation and Morrey's inequality
 \begin{align} \label{phi.C^1.alpha}
     \|\varphi\|_{C^{1,1-3/{p'}}}\lesssim \|g\|_{L^{p'} \cap L^1} \lesssim \|g\|_{p'},
 \end{align}
 where we used $\supp g \subset \overline{B_1(x)}$ in the last estimate. 
 Moreover, by linearity and scaling,
 \begin{align} \label{v_j.H^1}
     \|D v_j\|_2 \lesssim \frac{|F_j|}{r^{1/2}} + \frac{|T_j|}{r^{3/2}} \lesssim r^{3/2}.
 \end{align}
 Collecting  \eqref{g.v_j.tilde}, \eqref{phi.to.psi}, \eqref{phi.C^1.alpha} and \eqref{v_j.H^1} yields
 \begin{align}
     \left|\int_{\R^3} g \cdot (\tilde v_j -  v_j)\right| \lesssim r^{4-3/{p'}} \|g\|_{p'}.
 \end{align}
 Since $g \in L^{p'}(\R^3)$ with $\supp g \subset \overline{B_1(x)}$ was arbitrary, we conclude \eqref{v_j.tilde.v_j}.
\end{proof}

\begin{lem} \label{lemma:reflection method}
We have
\begin{align}
    \|\bar u_N - \sum_j v_j \|^2_{\dot H^1(\R^3)}  \lesssim \sum_i \| \sum_{j \neq i} \nabla  v_j\|^2_{L^2(\B_i)} \lesssim \phi^3  \label{v_j.u_N.H^1}.
\end{align}
\end{lem}
\begin{proof}
We introduce the short notation 
\begin{align}
    d_{ij} = \begin{cases}
        |X_i - X_j| &\quad \text{for } i \neq j,\\
        \dmin &\quad\text{for } i = j.
    \end{cases}
\end{align}
We will use the following estimates  taken from \cite[Lemma 4.8]{NiethammerSchubert19} on 
\begin{align}
    \mathcal S_n := \sup_i \sum_{j} \frac 1 {d_{ij}^n}
\end{align}
for $n \in \N$ and all $1 \leq j,k \leq N$:
\begin{align}
    \mathcal S_2 &\lesssim \frac{N^{1/3}}{\dmin^2} \lesssim N, \label{S_2} \\
    \mathcal S_4 &\lesssim \frac 1 {\dmin^4} \lesssim N^{4/3} \label{S_4}.
\end{align}

        We introduce the remainder $R_N:= \bar u_N - \sum_j v_j$, which solves 
\begin{align} \label{R_N}
\left\{\begin{array}{rl}
    - \Delta R_N + \nabla p = 0, \quad \dv R_N = 0 &\quad \text{in} \quad \R^3 \setminus \cup_i \B_i,\\
    D R_N= - \sum_{j \neq i} D v_j(x)  &\quad \text{in} \quad \B_i ~ \forall i, \\
    \int_{\partial \B_i} \Sigma[R_N,p] n  = 0 = \int_{\partial \B_i} (x - X_i) \times \Sigma[R_N,p] n & \quad \text{for all } i.
    \end{array}    \right.
\end{align}
By a standard argument, 
\begin{align} \label{est.R_N}
    \|\bar u_N - \sum_j v_j \|^2_{\dot H^1(\R^3)} \lesssim  \sum_i \| \sum_{j \neq i} \nabla  v_j\|^2_{L^2(\B_i)}.
\end{align}
Indeed, by \cite[Lemma 4.6]{HoeferLeocataMecherbet22} there exists a function $\varphi \in \dot H^1(\R^3)$ with $\dv \varphi = 0$ such that $D \varphi = D R_N$ in $\cup_i \B_i$ and 
\begin{align}
    \|\varphi\|_{\dot H^1(\R^3)} \lesssim \|D R_N\|_{L^2(\cup_i \B_i)}.
\end{align}
Testing \eqref{R_N} with $R_N - \varphi$ yields \eqref{est.R_N}.
This proves the first inequality in \eqref{v_j.u_N.H^1}.

Moreover, by \cite[Proposition 4.5]{HoeferLeocataMecherbet22},
for all $x \in \R^3 \setminus (X_j + \theta r \B)$
\begin{align}
    v_j(x) = \Phi(x-X_j) \cdot F_j + \nabla \Phi(x-X_j) : S_j + \mathcal  H_j(x-X_j), 
\end{align}
where
\begin{align}
    |\nabla \mathcal H_j(x)| \lesssim \frac {|F_j|r}{|x|^{3}} +  \frac {|S_j|r}{|x|^{4}},
\end{align}
and where $\Phi$ is the fundamental solution of the Stokes equation,
\begin{align}
    \Phi(x) = \frac 1 {8 \pi}\left( \frac 1 {|x|} + \frac{x \otimes x}{|x|^3} \right).
\end{align}
Strictly speaking, \cite[Proposition 4.5]{HoeferLeocataMecherbet22} only applies to $F_j = 0$ but the adaptation of the proof to $F_j \neq 0$ is straightforward.  

Hence, 
\begin{align} \label{split.sum.v_j}
\begin{aligned}
    \sum_i \| \sum_{i \neq j} \nabla  v_j\|^2_{L^2(\B_i)} \lesssim \sum_i \| \sum_{j \neq i} \nabla \Phi(\cdot-X_j) \cdot F_j \|^2_{L^2(\B_i)} &+ \sum_i \| \sum_{j \neq i} \nabla^2 \Phi(\cdot-X_j) : S_j \|^2_{L^2(\B_i)} \\
    &+ \sum_i \| \sum_{j \neq i} \nabla  \mathcal H_j\|^2_{L^2(\B_i)} .
\end{aligned}
\end{align}
The first and last right-hand side term are easily estimated by an explicit computation relying on the decay of $\Phi$  and the form of $F_j$ and $S_j$ from \eqref{F_j}--\eqref{S_j}. Firstly, by \eqref{S_2},
\begin{align} \label{sum.F_j}
    \sum_i \| \sum_{j \neq i} \nabla \Phi(\cdot-X_j) \cdot F_j \|^2_{L^2(\B_i)} & \lesssim r^6 \phi S_2^2 \lesssim \phi^3.
\end{align}
Secondly, using first \eqref{K} to bound $|X_i - X_j| \lesssim 1$ and then also  \eqref{S_4},
\begin{align} \label{sum.H_j}
    \sum_i \| \sum_{j \neq i} \nabla  \mathcal H_j\|^2_{L^2(\B_i)} & \lesssim \phi r^{8} S_4^2 \lesssim \phi^{11/3} \lesssim \phi^3,
\end{align}
since \eqref{K}--\eqref{assumption separation} imply $\phi \leq |k| \lesssim 1$.
Finally, the second right-hand side term  of \eqref{split.sum.v_j} satisfies 
\begin{align} \label{sum.S_j}
    \sum_i \| \sum_{j \neq i} \nabla^2 \Phi(\cdot-X_j) : S_j \|^2_{L^2(\B_i)} \lesssim \phi^3.
\end{align}
This follows by arguments as in \cite{Hoefer19}. We provide the details for the reader's convenience.
We write $K^N = \nabla^2 \Phi \1_{B_{\dmin/4}(0)}$. We introduce the function 
\begin{align}
    g = \frac 1  {|B_{\dmin/4}(0)|} \sum_j S_j \1_{B_{\dmin/4}(X_j)}.
\end{align}
Then, for all $x \in B_i$
\begin{align}
      \sum_{j \neq i} \nabla^2 \Phi(x-X_j)  S_j    &=
    \sum_j  \int_{B_{\dmin/2}(X_j)} (K^N(x-X_j) - K^N(x-y))g(y) \dd y \\
     & +  \fint_{B_{\dmin/4}(X_i)} \int_{\R^3} (K^N(x-y) - K^N(z-y))g(y) \dd y \dd z \\
     &+ \fint_{B_{\dmin/4}(X_i)} \int_{\R^3}  K^N(z-y)g(y) \dd y \dd z \\
     &=: A_1^i(x) + A_2^i(x) + A_3^i.
\end{align}
Since for all $x \in \B_i, y \in B_{\dmin/2}(X_j)$
\begin{align}
    |K^N(x-X_j) - K^N(x-y)| \lesssim \frac {\dmin}{d_{ij}^4},
\end{align}
we have thanks to \eqref{assumption separation} and \eqref{S_4}
\begin{align}
    |A_1^i| \lesssim  r^3 \sum_{j \neq i} \frac {\dmin}{d_{ij}^4} \lesssim \phi.
\end{align}
Similarly, we have $|A_2^i| \lesssim \phi$ such that
\begin{align}
    \sum_i \|A_1^i + A_2^i\|_{L^2(\B_i)}^2 \lesssim \phi^3.
\end{align}
Lastly, since $A_3^i$ is independent of $x$
\begin{align}
    \sum_i \|A_2^i\|_{L^2(\B_i)}^2 \lesssim r^3 \sum_i |A_2^i|^2 \lesssim \phi \int_{\R^3} | K^N \ast g|^2  \dd y \dd z \lesssim \phi \|g\|_{L^2}^2 \lesssim \phi^3,
\end{align}
where we used in the second last estimate that $K$ is a Calderon-Zygmund operator. 
Gathering these estimates yields \eqref{sum.S_j}.

Inserting \eqref{sum.F_j}, \eqref{sum.H_j} and \eqref{sum.S_j} into \eqref{split.sum.v_j}, we deduce
\begin{align}
      \sum_i \| \sum_{i \neq j} \nabla  v_j\|^2_{L^2(\B_i)} \lesssim \phi^3.
\end{align}
In combination with \eqref{est.R_N}, this yields \eqref{v_j.u_N.H^1}.
\end{proof}

\begin{prop} \label{prop:angular.velocity}
We have
\begin{equation}
    \label{inequality reflection rotation speed}
    \frac{1}{N} \sum_{i=1}^N \left|\nabla \bar u_N(X_i)  \bar \xi_i - \bar \sigma_i \gamma P_{\bar \xi_i^\perp} \bar H(X_i) \right|^2 \lesssim   \phi^2.
\end{equation}
\end{prop}
\begin{proof} 
We observe that $D \bar u_N = 0$ in $\B_i$ implies $\nabla \bar u_N (X_i) = \fint_{\B_i} \nabla \bar u_N$ which allows us to write
\begin{align*}
     \frac{1}{N} \sum_{i=1}^N \left| \bar \sigma_i \gamma P_{\bar \xi_i^\perp} \bar H(X_i) - \nabla \bar u_N \bar \xi_i \right|^2  &\lesssim
   \frac{1}{N} \sum_{i} \Big|\bar \sigma_i \gamma P_{\bar \xi_i^\perp} \bar H(X_i) - \fint_{\B_i} \nabla v_i \bar \xi_i\Big|^2\\
    &+  \frac{1}{N} \sum_{i} \Big|\fint_{\B_i} \Big(\sum_{j} \nabla v_j -  \bar u_N  \Big) \Big|^2 \\  &+   \frac{1}{N} \sum_{i} \Big| \sum_{j \neq i} \fint_{\B_i} \nabla v_j \bar \xi_i\Big|^2 .
\end{align*}
Regarding the first right-hand side term, we observe that (up to a translation), $v_j$ is the solution $v_{T_j,F_j}$ to \eqref{v_T,F}. 
Hence, by \eqref{angular.single.2} 
\begin{align*}
   \bar \sigma_i \gamma P_{\bar \xi_i^\perp} \bar H(X_i) = \fint_{\B_i} \nabla v_i \bar \xi_i.
\end{align*} 

Moreover, we estimate
\begin{align*}
     &\frac{1}{N} \sum_{i} \Big|\fint_{\B_i} \nabla \Big(\sum_{j} \nabla v_j -  \bar u_N  \Big)  \bar \xi_i\Big|^2 +   \frac{1}{N} \sum_{i} \Big| \sum_{j \neq i} \fint_{\B_i} \nabla v_j \bar \xi_i\Big|^2 \\
     &\leq \frac 1 \phi \left(\|\nabla \bar u_N - \sum_j \nabla v_j\|_{L^2(\cup_i \B_i)}^2 + \sum_{i} \| \sum_{j \neq i} \nabla v_j\|_{L^2( \B_i)}^2 \right) \\
     &\lesssim \phi^2
\end{align*}
due to Lemma \ref{lemma:reflection method}.
\end{proof}

\section{Effective orientation dynamics for fast spin flip} \label{sec:relaxation}

We recall the definition of $m^{\pm}$  from \eqref{m_pm} and denote

where $ \xi^\ast_1(t), \dots, \xi^\ast_N(t)$ are the orientations defined through the following ODEs
\begin{align}
\label{equation orientation one particle}
   \left\{ \begin{array}{l}
    \dot \xi^\ast_i(t) = \gamma m^0(t,X_i,\xi_i^\ast(t)) P_{(\xi^\ast_i(t))^\perp} H(t,X_i) =: m_i^{\ast,0}(t) V_i^\ast(t,\xi_i^\ast(t)),\\
    \xi^\ast_i(0) = \xi_{i}^0 .
    \end{array} \right.
\end{align}
with the notation
\begin{align}
    m_i^{\ast,\pm}(t) = m^{\pm}(t,X_i,\xi_i^\ast(t)), \qquad    m_i^{\ast,0}(t) =  m^0(t,X_i,\xi_i^\ast(t)).
\end{align}
Note that $m_i^{\ast,0}$ and $V_i^\ast$ only depend on $\xi_i^\ast$ such that the ODEs are decoupled.
Moreover, we denote 
\begin{align} \label{def.D_i}
    D_i(t) = \left|\P[\sigma_i(t) = +1] - m_i^{\ast,+}(t)\right| .
\end{align}


\begin{prop}
\label{proposition Estimate for the fast spinning regime}
For any $t< T$ and $1 \leq i \neq j \leq N$, 
\begin{align} 
\label{control wasserstein microscopic}
     \E [|\xi(t) -\xi^\ast(t)|^2]  &\lesssim  (\eps  + \phi^2) Nt, \\ \label{control wasserstein microscopic p_eps}
     \sum_i D_i^2(t)
     &\lesssim  N\left(e^{\tfrac{-ct}\eps}    + \eps + \phi^2\right), \\
     |\Cov(\sigma_i,\sigma_j)|^2 &\lesssim  e^{\tfrac{-ct}\eps} +  \frac 1 \eps \int_0^t e^{\tfrac{-c(t-s)}\eps}\Big(\E[|\xi_i(s) - \xi_i^\ast(s)|^2] + \E[|\xi_j(s)- \xi_j^\ast(s)|^2]\Big) \dd s.\qquad \label{est.Cov}
\end{align}
\end{prop}

\begin{rem} 
    The factor $N$ simply arises from the dimension of the state space $(S^2)^N \times \spinstate^N$. 

    Naively, one might expect the error with respect to $\eps$ to be of order $\eps^2$. However, owing to the central limit theorem, the expected time the spin $\sigma_i$ spends in state $+1$ during a unit time interval is $m^+_i$ (assuming $m^+_i$ was constant) but   has fluctuations of order $\sqrt{\eps}$. Therefore, the average velocity in that time interval will have fluctuations around $m^0_i V^\ast_i$ of the same order. This explains the error of order $\eps$ in \eqref{control wasserstein microscopic}. 

\end{rem}




\begin{proof} \textbf{Step 1:} \emph{Proof of \eqref{control wasserstein microscopic}.} \\
\textbf{Substep 1.1:} \emph{Structure of the proof.}
We have
\begin{align} \label{W}
   \mathcal{W}(t) := \E [|\xi(t) -\xi^\ast(t)|^2]  =\int_{{(\S^2)^N} \times \{\pm 1\}^N} |\bar \xi -  \xi^\ast(t)|^2  p_\eps(t,\d \bar \xi,\d \bar \sigma).
\end{align}

The idea is to derive inequalities for the time derivative of $\mathcal W$ and conclude by Gronwall's inequality. However,  the relaxation mechanism is due to the fast spin flip which only occurs through the second time derivative.

More precisely, to leading order, the first time derivative of $\mathcal W$ only reflects the (deterministic) ODEs for the particle orientation given through \eqref{ODE.Xi} and \eqref{equation orientation one particle}. Hence, we need an estimate for $V_i - m_i^{\ast,0} V^\ast_i$. Even though we already have good estimates for $V_i - \sigma_i V_i^\ast$ through Proposition \ref{prop:angular.velocity}, $\sigma_i - m_i^{\ast,0}$ will not be small. However, we expect smallness of $\langle \sigma_i \rangle - m_i^{\ast,0}$ where $\langle \sigma_i \rangle$ denotes the average with respect to $p_\eps$. Indeed, this is the relaxation mechanism through the fast spin flip.

To make this reasoning rigorous, we will show that
\begin{align}
    \dot{\mathcal{W}}  = 2(S + R_1 + R_2),
\end{align}
where $S$ satisfies
\begin{align}
     \frac{\d}{\d t} |S| \leq - \frac{c}{\eps} |S| + |R_3| + \frac{|R_4|}{\eps},
\end{align}
and where $R_k$, $1\leq k \leq 4$ are remainder terms.

Through application of Gronwall's inequality, we will conclude \eqref{control wasserstein microscopic}. \\[2mm]
\textbf{Substep 1.2:} \emph{Time derivative of $\mathcal{W}$.}
We note that \eqref{generator micro system} means that for any $\varphi \in C^1([0,T] \times (\S^2)^N \times \spinstate^N)$, we have for any $t \in [0,T]$
\begin{align}
    \int_{{(\S^2)^N} \times \{\pm 1\}^N} \varphi(t,\cdot) \dd p_\eps(t,\cdot) 
    &=  \int_{{(\S^2)^N} \times \{\pm 1\}^N} \varphi(0,\cdot) \dd p_\eps(0,\cdot) \\
    & \hspace{-3cm}+ \int_0^t\int_{{(\S^2)^N} \times \{\pm 1\}^N} \left((\partial_t + V \cdot \nabla_{\bar \xi})  \varphi  +
     \frac 1 \eps \sum_i \left(\lambda_i  ( \tau_i \varphi -  \varphi) \right)  \right)
   \dd p_\eps(s,\cdot) \dd s. 
\end{align}
Here, we used that the operator $\tau_i$ that acts on functions via \eqref{tau_i.operator} satisfies $\tau_i^{-1} = \tau_i$.
In particular,
\begin{align*}
    \frac {\dd}{\dd t} \int_{{(\S^2)^N} \times \{\pm 1\}^N} \varphi(t,\cdot) \dd p_\eps(t,\cdot) = \int_{{(\S^2)^N} \times \{\pm 1\}^N} \left((\partial_t + V \cdot \nabla_{\bar \xi})  \varphi  +
     \frac 1 \eps \sum_i \left(\lambda_i  ( \tau_i \varphi -  \varphi) \right) \right)  \dd p_\eps(t,\cdot).
\end{align*}
In the following, we omit for shortness all $t$ dependencies, and write an index $i$ instead of writing the dependence on $X_i$. E.g. $H_{i}$ is short for $H(t,X_i)$ and $\lambda^{\bar \sigma_i}_i(\bar \xi_i)$ is short for $\lambda^{\bar \sigma_i}(t,X_i,\bar \xi_i(t))$.
Using that $\tau_i \varphi -  \varphi = 0$ for all $\varphi$ which are independent of $\bar \sigma$, we find
\begin{align} \label{ODE.mW}
\begin{aligned}
     \dot{\mathcal{W}}  &=  2 \sum_i \int_{{(\S^2)^N} \times \{\pm 1\}^N}  (\bar \xi_i -  \xi_i^\ast) \cdot \Big(V_i(\bar \xi,\bar\sigma) -  m^0_i(\xi^\ast_i) V^\ast_i(\xi^\ast_i)\Big) \dd p_\e\\
     & =2 \sum_i \int_{{(\S^2)^N} \times \{\pm 1\}^N}  (\bar \xi_i -  \xi^\ast_i) \cdot \Big(( \bar \sigma_i - m^0_i(\xi^\ast_i))  V^\ast_i(\xi^\ast_i)\Big) \dd p_\e\\
    & + 2 \sum_i \int_{{(\S^2)^N} \times \{\pm 1\}^N}  (\bar \xi_i -  \xi^\ast_i) \cdot \Big(\bar \sigma_i (V^\ast_i(\bar \xi_i) -  V^\ast_i( \xi^\ast_i)\Big) \dd p_\e\\
    & +  2 \sum_i \int_{{(\S^2)^N} \times \{\pm 1\}^N}  (\xi_i -  \xi^\ast_i) \cdot \Big(V_i(\bar \xi,\bar\sigma) -  \bar \sigma_i  V^\ast_i(\bar \xi_i)\Big) \dd p_\e\\
    & := 2(S + R_1 + R_2).
    \end{aligned}
\end{align}
We recall the definition of $V$ from \eqref{ODE.Xi}.  Note that from Proposition \ref{prop:angular.velocity}, we have for all $(\bar \xi,\bar \sigma) \in (\S^2)^N \times \spinstate^N$
\begin{align*}
    \frac{1}{N} \sum_{i=1}^N \left|V_i(\bar \xi,\bar \sigma) -  \bar \sigma_i V_i^\ast(\bar \xi_i) \right|^2  \lesssim \phi^2.
\end{align*}

Hence, 
\begin{align} \label{est.R_2}
\begin{aligned}
    |R_2| & \lesssim  \int_{{(\S^2)^N} \times \{\pm 1\}^N} |\bar \xi - \xi^\ast(t)|^2 \dd p_\e +  \sup_{(\bar \xi,\bar \sigma)} \sum_{i=1}^N \left|V_i(\bar \xi,\bar \sigma) -  \bar \sigma_i \gamma P_{\bar \xi_i^\perp} H_{i} \right|^2
     \lesssim \mathcal{W}(t) + N \phi^2 .
    \end{aligned}
\end{align}\\[2mm]
\textbf{Substep 1.3:} \emph{Time derivative of $S$.}
 We compute 
\begin{align*}
    \dot S(t) & = \sum_i \int_{{(\S^2)^N} \times \{\pm 1\}^N} (\partial_t +  V_i(\bar \xi,\bar \sigma) \cdot \nabla_{\bar \xi_i})  \left[  (\bar \xi_i -  \xi^\ast_i) \cdot \Big(( \bar \sigma_i - m^0_i(\xi^\ast_i)) V_i^\ast(\xi_i^\ast)\Big) \right]  \dd p_\e \\
    & - \frac 2 \eps \sum_i \int_{{(\S^2)^N} \times \{\pm 1\}^N} \bar \sigma_i \lambda^{\bar \sigma_i}_i(\bar \xi_i)(\bar \xi_i - \xi^\ast_i(t)) \cdot  V_i^\ast( \xi^\ast_i) \dd p_\e.
\end{align*}
We observe that for any $i \in \{1, \dots, N\}$, the definition $m^0$ in \eqref{m^0} together with \eqref{m_pm} directly implies
\begin{align}
    \bar \sigma_i \lambda_i^{\bar \sigma_i}(\bar \xi_i) = (\lambda^+_i(\bar \xi_i) + \lambda^-_i(\bar \xi_i)) \frac{\bar \sigma_ i - m^0_i(\bar \xi_i)}{2},
\end{align}
such that 
\begin{align}
     & -2 \int_{{(\S^2)^N} \times \{\pm 1\}^N} \bar \sigma_i \lambda^{\bar \sigma_i}_i(\bar \xi_i)(\bar \xi_i - \xi^\ast_i(t)) \cdot  V_i^\ast( \xi^\ast_i) \dd p_\e \\
     &=  (\lambda^+_i + \lambda^-_i)(\xi^\ast_i) \int_{{S^2}^N \times \{\pm 1\}^N} (m^0_i(\xi^\ast_i) - \bar \sigma_i)) (\bar \xi_i - \xi^\ast_i(t)) \cdot V_i^\ast( \xi^\ast_i)  \dd p_\eps \\
     & + \int_{\S^2 \times \{\pm 1\}} \left[(\lambda^+_i + \lambda^-_i)(\bar \xi_i)( m^0_i(\bar \xi_i) - \sigma_i) -(\lambda^+_i + \lambda^-_i)(\xi^\ast_i)( m_i^0(\xi^\ast_i) - \bar \sigma_i )  \right] (\bar \xi_i - \xi^\ast_i) \cdot  V_i^\ast( \xi^\ast_i)   \dd p_\eps.
\end{align}
We observe that by definition of $\lambda^{\pm}$ in \eqref{def:intensities lambda}, we have
\begin{align} \label{lambda.below}
\lambda^{+} + \lambda^- \geq 1.
\end{align}
Hence, (for $S \neq 0$) 
\begin{align} \label{est.S}
     \frac{\d}{\d t} |S| \leq - \frac{1}{\eps} |S| + |R_3| + \frac{|R_4|}{\eps},
\end{align}
where 
\begin{align*}
    R_3 :=\sum_i \int_{{(\S^2)^N} \times \{\pm 1\}^N} (\partial_t +  V_i(\bar \xi,\bar \sigma) \cdot \nabla_{\bar \xi_i})  \left[  (\bar \xi_i -  \xi^\ast_i) \cdot \Big(( \bar \sigma_i - m^0_i(\xi^\ast_i)) V_i^\ast(\xi_i^\ast)\Big) \right]  \dd p_\e,
\end{align*}
and
\begin{multline*}
     R_4 :=  \int_{\S^2 \times \{\pm 1\}} \left[(\lambda^+_i + \lambda^-_i)(\bar \xi_i)( m^0_i(\bar \xi_i) - \sigma_i) -(\lambda^+_i + \lambda^-_i)(\xi^\ast_i)( m_i^0(\xi^\ast_i) - \bar \sigma_i )  \right]\\ (\bar \xi_i - \xi^\ast_i) \cdot  V_i^\ast( \xi^\ast_i)   \dd p_\eps.
\end{multline*}
We observe that 
\begin{align}
    |R_1| +   |R_4| &\lesssim  \mathcal{W}, \label{est.R_1.R_4} \\
   |R_3| &\lesssim N.  \label{est.R_3}
\end{align}\\[2mm]
\textbf{Substep 1.4:} \emph{Conclusion.}
We observe that $\xi_i(0) = \xi_\ast(0)$ implies $S(0) = )$. Hence, by  Gronwall's inequality, estimates \eqref{est.S},  \eqref{est.R_1.R_4} and \eqref{est.R_3} yield  
\begin{align}
    |S(t)| \lesssim  \int_0^t e^{-\frac {c(t-s)} \eps} \left(\frac{  \mathcal{W}(s) }{\eps} + N\right) \dd s.
\end{align}
Inserting that into \eqref{ODE.mW} together with the estimates \eqref{est.R_2}, \eqref{est.R_1.R_4} and $\mathcal W(0) = 0$, Young's inequality yields
\begin{align}
    \mathcal{W}(t) &\lesssim   \int_0^t (N\phi^2  + \mathcal W(s)) \dd s + \int_0^t \int_0^s e^{-\frac {c(s-s')} \eps} \left(\frac{  \mathcal{W}(r) }{\eps} + N\right) \dd s' \dd s   \\
    & \lesssim  \eps^2  +  \int_0^t \left(  \mathcal{W}(s') + \eps N + N\phi^2\right) \dd s.
\end{align}
By Gronwall's inequality we conclude
\begin{align}
    \mathcal{W}(t) \lesssim  (\eps  + \phi^2) Nt  . 
\end{align}\\[2mm]
\textbf{Step 2:} \emph{Proof of \eqref{control wasserstein microscopic p_eps}.}
We compute, 
\begin{align}
 \frac {\d}{\d t} \P[\sigma_i = +1] &= \frac {\d}{\d t} \int_{{(\S^2)^N} \{\pm 1\}^N} \1_{\bar \sigma_i =1} \dd p_\eps \\
 &= \frac 1 \eps \int_0^t \int_{{(\S^2)^N} \times \{\pm 1\}^N} \lambda_i  (\1_{\bar \sigma_i = -1} - \1_{\bar \sigma_i = 1}) \dd p_\eps\\
&= \frac 1 \eps \int_0^t \int_{{(\S^2)^N} \times \{\pm 1\}^N} \lambda^-_i(\bar \xi_i)  \1_{\bar \sigma_i =-1} - \lambda^+_i(\bar \xi_i) \1_{\bar \sigma_i = 1}\dd p_\e \\
& = \frac 1 \eps  \bigg(\int_{{(\S^2)^N} \times \{\pm 1\}^N} (\lambda^-(\bar \xi_i) - \lambda^-(\xi^\ast_i))  \1_{\bar \sigma_i =-1}  - (\lambda^+(\xi_i) - \lambda^+(\xi^\ast_i)) \1_{\bar \sigma_i =1} \dd p_\eps \\
& + \lambda^-(\xi^\ast_i) - (\lambda^+(\xi^\ast_i) + \lambda^-(\xi^\ast_i))  \int_{{(\S^2)^N} \times \{\pm 1\}^N} \1_{\bar \sigma_i =1} \dd p_\eps  \bigg) ,
\end{align}
where we used in the last identity that $p_\eps(\{\bar \sigma_i = +1\}) = 1 - p_\eps(\{\bar \sigma_i = -1\})$. Furthermore, since $m_i^{\ast,+} = \frac{\lambda_i^-(\xi^\ast_i)}{\lambda_i^+(\xi^\ast_i) + \lambda_i^-(\xi^\ast_i)}$ and using \eqref{lambda.below}, we deduce
\begin{align}
    \frac {\dd }{\dd t} D_i^2& = - 2\frac { \lambda^+(\xi^\ast_i) + \lambda^-(\xi^\ast_i)} \eps  D_i^2 +  2 \left(\int_{{(\S^2)^N} \{\pm 1\}^N} \1_{\bar \sigma_i =1} \dd p_\eps   -m_i^{\ast,+} \right) \\ & \times  \left(\frac 1 \eps \bigg(\int_{{(\S^2)^N} \times \{\pm 1\}^N} (\lambda^-(\bar \xi_i) - \lambda^-(\xi^\ast_i))  \1_{\bar \sigma_i =-1}  - (\lambda^+(\xi_i) - \lambda^+(\xi^\ast_i)) \1_{\bar \sigma_i =1} \dd p_\eps  +\frac {\dd}{\dd t} m_i^{\ast,+}\right) \\
    &\leq  - \frac {2} \eps  D_i^2 + C D_i \left( 1 + \frac 1 \e \int_{{(\S^2)^N} \times \{\pm 1\}^N} | \bar \xi_i - \xi^\ast_i| \dd p_\e \right).
\end{align}
Let $E = \sum_i D_i^2$. Then, by Young's inequality 
\begin{align}
    \frac {\dd }{\dd t} E \leq - \frac {c} {2\eps}  E + C \left(\frac { \E [|\xi(t) -\xi^\ast(t)|^2]} \eps  + N \eps \right).
\end{align}
By  Gronwall's inequality and \eqref{control wasserstein microscopic}, we find 
\begin{align}
     E(t) \leq E(0) e^{-\frac {ct} \eps} + C N(\eps + \phi^2) \leq N e^{-\frac {ct} \eps} + C N(\eps + \phi^2) .
\end{align}
and hence
\begin{align}
     \sum_i D_i(t) \leq \sqrt N \sqrt{E(t)} \lesssim  N  e^{-\frac {ct} \eps} +  N(\sqrt \eps +  \phi).
\end{align}
\\[2mm]
\textbf{Step 3:} \emph{Proof of \eqref{est.Cov}.}
We introduce the short notation $\langle \varphi \rangle := \E[\varphi]$. Then
\begin{align}
    \frac{\d}{\d t} \langle \sigma_i\sigma_j \rangle &= -\frac 2 \eps \int_{(\S^2)^N \times  \{\pm 1\}^N} \bar \sigma_i \bar \sigma_j (\lambda_i + \lambda_j) \dd p_\eps \\
    &= -\frac 2 \eps \int_{(\S^2)^N \times  \{\pm 1\}^N} \bar \sigma_i \bar \sigma_j (\lambda^{\bar \sigma_i}(\xi^\ast_i)  + \lambda^{\bar \sigma_j}(\xi^\ast_j)) \dd p_\eps + \frac 1 \eps \O\Big(\E[|\xi_i - \xi_i^\ast|] + \E[|\xi_j - \xi_j^\ast|]\Big),
\end{align}
and similarly
\begin{align}
    \frac{\dd}{\dd t} (\langle \sigma_i \rangle \langle \sigma_j \rangle) &=  - \frac{2}{\e} \langle \sigma_i \rangle \int_{(\S^2)^N \times \{ \pm 1\}^N} \bar \sigma_j \lambda^{ \bar \sigma_j}(\xi^\ast_j) \dd p_\e - \frac{2}{\e} \langle \sigma_j \rangle \int_{(\S^2)^N \times \{ \pm 1\}^N} \bar \sigma_i \lambda^{\bar \sigma_i}(\xi^\ast_i) \dd p_\e \\
    & + \frac 1 \eps  \O\Big(\E[|\xi_i - \xi_i^\ast|] + \E[|\xi_j - \xi_j^\ast|]\Big).
\end{align}
In order to compute the right-hand side integrals in these two identities we observe that 
\begin{align}
    \P(\sigma_i = \pm 1) = \frac{1 \pm \langle \sigma_i \rangle}2
\end{align}
and 
\begin{align*}
\P(\sigma_i = 1, \sigma_j = 1) & = \frac{1}{4} \Big( 1+\langle \sigma_i \sigma_j \rangle  + \langle \sigma_i \rangle +\langle \sigma_j \rangle\Big)\\
  \P(\sigma_i = 1, \sigma_j =-1) & = \frac{1}{4} \Big(1 -\langle \sigma_i \sigma_j \rangle  + \langle \sigma_i \rangle -\langle \sigma_j \rangle\Big)\\
   \P(\sigma_i =-1, \sigma_j = 1) & = \frac{1}{4} \Big(1 -\langle \sigma_i \sigma_j \rangle  -\langle \sigma_i \rangle +\langle \sigma_j \rangle\Big)\\
    \P(\sigma_i =-1, \sigma_j =-1)& = \frac{1}{4} \Big(1+ \langle \sigma_i \sigma_j \rangle  - \langle \sigma_i \rangle -\langle \sigma_j \rangle\Big).
\end{align*}
A tedious but straightforward calculation then leads to
\begin{align*}
    \frac{\dd}{\dd t} \Cov(\sigma_i,\sigma_j) &= -  \frac{1}{\e} \left(\lambda^+(\xi_i^\ast) + \lambda^-(\xi_i^\ast) + \lambda^-(\xi_j^\ast)  + \lambda^+(\xi_j^\ast) \right) \Cov(\sigma_i,\sigma_j) \\
    &+ \O\Big(\E[|\xi_i - \xi_i^\ast|] + \frac 1 \eps  \E[|\xi_j - \xi_j^\ast|]\Big) .
\end{align*}
Solving this ODE and using \eqref{lambda.below} yields
\begin{align}
     |\Cov(\sigma_i,\sigma_j)| &\lesssim  e^{\tfrac{-ct}\eps} +  \frac 1 \eps \int_0^t e^{\tfrac{-c(t-s)}\eps}\Big(\E[|\xi_i(s) - \xi_i^\ast(s)|] + \E[|\xi_j(s)- \xi_j^\ast(s)|]\Big) \dd s.
\end{align}
Taking the square and using the Cauchy-Schwarz inequality and $\int_0^t e^{\tfrac{-c(t-s)}\eps} \dd s \lesssim \eps$ yields \eqref{est.Cov}.
\end{proof}

\section{Proof of the main result}
 \label{sec:proof.main}

\begin{proof}[Proof of Theorem \ref{main theorem}] 
\textbf{Step 1:} \emph{Proof of \eqref{f_N,f}.} \\
We define
\begin{align} \label{f^ast_N}
    f_N^\ast(t) := \sum_{i=1}^N \delta_{(X_i,\xi^\ast_i(t))},
\end{align}
where $\xi_i^\ast$ are defined in \eqref{equation orientation one particle}. 
Then, 
\begin{align*}
        \mathbb{E} \big[W_2^2(f_N, f )\big] \lesssim  \mathbb{E} \big[W_2^2(f_N, f_N^\ast)\big] + \mathbb{E} \big[W_2^2(f_N^\ast, f) \big].
\end{align*}
By Proposition \ref{proposition Estimate for the fast spinning regime} we have

\begin{align} \label{W_2.f_N.f_N^ast}
    \mathbb{E} \big[W_2^2(f_N(t), f_N^\ast(t))\big]  & \leq \mathbb{E} \bigg[ \frac{1}{N} \sum_{i=1}^N |\xi_i(t) - \xi^\ast_i(t)|^2\bigg] 
    \lesssim (\eps  + \phi^2) t.
\end{align}
It remains to estimate $W_2(f_N^\ast, f)$. We note that both $f_N^\ast$ and $f$ are (distributional) solutions to the (deterministic) continuity equation
\eqref{eq:f}.
Hence, by a standard stability estimate
\begin{align} \label{W_2.f_N^ast}
   W_2(f_N^\ast(t), f(t)) \leq W_2(f_N^0,f^0) e^{2 L t},
\end{align}
where $L \lesssim 1$ is  the Lipschitz constant of $(x,\zeta) \mapsto \gamma m_0(t,x,\zeta) \gamma P_{\zeta^\perp} H(t,x) $. The original stability estimate goes back to \cite{Dobroushin79} for the $1$-Wasserstein distance. The proof can also be found in \cite[Theorem 3.3.3]{Golse16} and its adaptation to the $2$-Wasserstein distance is straightforward.
Combining \eqref{W_2.f_N.f_N^ast} and \eqref{W_2.f_N^ast} yields \eqref{f_N,f}.
\\[2mm]
\textbf{Step 2:} \emph{Proof of \eqref{h_N,h}}. \\
We now turn to the analysis of $\mathbb{E} \big[W_2^2( h_N(t), g(t) )\big]$. 
We introduce the following intermediate measures on $(\R^3 \times \S^2 \times \{\pm 1\})^N$:
\begin{align*}
     h_N^1(t)&:= \frac{1}{N} \sum_{i} \delta_{(X_i, \xi^\ast_i(t), \sigma_i(t))} , \\
     h_N^2(t)&:= \frac{1}{N} \sum_{i} \delta_{(X_i, \xi^\ast_i(t))}  \big(m^+(t,X_i,\xi_i^\ast(t)) \delta_{1} +m^-(t,X_i,\xi_i^\ast(t))  \delta_{-1}\big),
\end{align*} 
and we write 
\begin{align} \label{split.W_2.h_N}
    W_2^2( h_N, h )\lesssim  W_2^2( h_N,  h_N^1 )+ W_2^2( h_N^1,  h_N^2 ) +  W_2^2( h_N^2, h ).
\end{align}
\textbf{Substep 2.1:} \emph{Estimate of $W_2^2( h_N, h_N^1 )$ and $W_2^2( h_N^2, h )$.} \\ 
We observe that by \eqref{control wasserstein microscopic}
\begin{align} \label{W_2.h_N^1}
    \E \big[W_2^2( h_N,  h_N^1 )\big] \leq \frac 1 N \E[|\xi-\xi^\ast|^2] \lesssim \eps + \phi^2.
\end{align}

Let $\gamma$ be an optimal transport plan for the $(f_N^\ast,f)$, where $f_N^\ast$ is defined as in \eqref{f^ast_N}. Then, we define
\begin{align}
    \bar \gamma(\d x_1,\d x_2,\d \zeta_1, \d \zeta_2, \d \varsigma_1, \d \varsigma_2) &:= \gamma(\d x_1,\d x_2,\d \zeta_1, \d \zeta_2) \nu(x_1,x_2,\zeta_1,\zeta_2)(\d \zeta_1, \d \zeta_2),
\end{align}
where
\begin{align}
    \nu(x_1,x_2,\zeta_1,\zeta_2) :=& \min\Big\{m^+(t,x_1,\zeta_1),m^+(t,x_2,\zeta_2)\Big\} \delta_{+1} \otimes \delta_{+1} \\
    &+ \min\Big\{m^-(t,x_1,\zeta_1),m^-(t,x_2,\zeta_2) \Big\} \delta_{-1} \otimes \delta_{-1} \\
    & + \max\Big\{0, m^+(t,x_1,\zeta_1) - m^+(t,x_2,\zeta_2)\Big\} \delta_{+1} \otimes \delta_{-1} \\
   & + \max\Big\{0,m^-(t,x_1,\zeta_1) - m^-(t,x_2,\zeta_2)\Big\} \delta_{-1} \otimes \delta_{+1}.
\end{align}
Since $m^+ + m^- = 1$, $\bar \gamma$ is a transport plan for $(h_N^2,h)$ and we get
\begin{align} \label{W_2.h_N^2}
\begin{aligned}
    W_2^2(h_N^2,h) &\leq \int_{(\R^3 \times \S^2 \times \spinstate)^2} |x_1 - x_2|^2 + |\zeta_1 - \zeta_2|^2 + |\varsigma_1 - \varsigma_2|^2 \dd \bar \gamma  \\
    &= 
    \int_{(\R^3 \times \S^2)^2} |x_1 - x_2|^2 + |\zeta_1 - \zeta_2|^2 + 4 |m^+(t,x_1,\zeta_1) - m^+(t,x_2,\zeta_2)| \dd \gamma \\
    & \lesssim W_2^2(f_N^\ast,f) + \|\nabla_{x,\zeta} m^+\|_{L^\infty_{t,x,\zeta}} \bigg(\int|x_1 - x_2|^2 + |\zeta_1 - \zeta_2|^2 \dd \gamma \bigg)^{1/2}\\
    & \lesssim  W_2^2(f_N^\ast,f) + W_2(f_N^\ast,f) \lesssim W_2^2(f_N^0,f^0) + W_2(f_N^0,f^0) \lesssim W_2(f_N^0,f^0),     
\end{aligned}
\end{align}
where we used \eqref{W_2.f_N^ast} in the second last estimate and $W_2(f_N^0,f^0) \leq \diam(K \times \S^2) \lesssim 1$ in the final estimate. 
\\[2mm]
\textbf{Substep 2.2:} \emph{Estimate of $W_2^2( h_N^1, h_N^2 )$.} \\
For $\delta > 0$ to be chosen later, we first introduce a covering of $K \times \S^2$ by  essentially disjoint sets $Q \in \mathcal Q$  such that $\diam(Q) \leq \delta$. We recall that  $K$ is as the set in \eqref{K} and observe that we can choose $\mathcal Q$ such that the number of sets needed for the covering satisfies
\begin{align} \label{number.cubes}
    |\mathcal Q| \lesssim \delta^{-5}.
\end{align}
For every cube $Q \in \mathcal{Q}$, 
we introduce the  probability measures
\begin{align*}
    & h_{N,Q}^{1,\pm} := \frac{1}{N_{Q}^{1,\pm}} \sum_{(X_i,\xi_i^0) \in Q} \delta_{(X_i,\xi_i^\ast(t), \pm 1)} \1_{\{\sigma_i = \pm 1\}},\\
    & h_{N,Q}^{2,\pm} := \frac{1}{  N_{Q}^{2,\pm}} \sum_{(X_i,\xi_i^0) \in Q} \delta_{(X_i,\xi^\ast_i(t), \pm 1)} m^\pm(t,X_i,\xi_i^\ast(t)),
\end{align*}
with
\begin{align*}
     N^{1,\pm}_{Q}&:= \# \{(X_i,\xi_i^0) \in Q :   \sigma_i(t) = \pm 1  \},\\
     N^{2,\pm}_{Q}&:= \sum_{(X_i,\xi_i^0) \in Q} m^\pm(t,X_i,\xi_i^\ast(t)),  \\
    N_{Q} &:= \# \{(X_i,\xi_i^0) \in Q\} = N^+_{Q} + N^-_{Q} = N^{+,2}_{Q} +  N^{2,-}_{Q},
\end{align*}
such that for $i=1,2$
\begin{align}
    h_N^i &= \sum_{Q \in \mathcal Q} \bigg(\frac{N_{Q}^{i,+}}{N} h_{N,Q}^{i,+} + \frac{N_{Q}^{i,-}}{N} h_{N,Q}^{i,-}\bigg).
\end{align}
These notations allow us to introduce the following transport plan between $h_N^1$ and $ h_N^2$:
\begin{align*}
\gamma:=   \sum_{Q \in \mathcal{Q}}  \Big[ &\tfrac{\min\{ N_{Q}^{1,+},  N_{Q}^{2,+}\}}{N} (h_{N,Q}^{1,+} \otimes  h_{N,Q}^{2,+}) + \tfrac{\min\{N_{Q}^{1,-}, N_{Q}^{2,-}\}}{N} (h_{N,Q}^{1,-} \otimes  h_{N,Q}^{2,-}) \\ &+  \tfrac{\max\{0,N_{Q}^{1,+} - N_{Q}^{2,+}\}}{N} (h_{N,Q}^{1,+} \otimes  h_{N,Q}^{2,-}) + \tfrac{\max\{0,N_{Q}^- - N_{Q}^{2,-}\}}{N} (h_{N,Q}^{1,-} \otimes  h_{N,Q}^{2,+}) \Big].
\end{align*}
In particular, observing that $N_{Q}^{1,+} -N_{Q}^{2,+} = - (N_{Q}^{1,-} - N_{Q}^{2,-})$, we get
\begin{align} \label{W_2.h_N.bar}
    \begin{aligned}
    W_2^2( h_N^1, h_N^2 ) & \leq \int_{(\R^3 \times \S^2 \times \spinstate)^2} |x_1  - x_2|^2 + |\zeta_1 - \zeta_2|^2  + |\varsigma_1 - \varsigma_2|^2 \dd \gamma 
    \\
    &\lesssim \delta^2 + \frac{1}{N}  \sum_{Q \in \mathcal Q} \left| N_{Q}^{1,+} - N_{Q}^{2,+} \right|.
    \end{aligned}
\end{align}
Here we used 
\begin{align} \label{zeta.stability}
   |x_1 - x_2| +  |\zeta_1 - \zeta_2|  \lesssim \delta \quad \text{for all } (x_1,x_2,\zeta_1,\zeta_2,\varsigma_1,\varsigma_2) \in \supp \gamma.
\end{align} 
Indeed, the estimate for $|x_1 - x_2|$ immediately follows from $\diam(Q) \leq \delta$. Moreover, for such $\zeta_1,\zeta_2$ there exist $1 \leq i_1,i_2 \leq N$ such that $|\xi_{i_1}^0 - \xi_{i_2}^0| \leq \delta$ and $\zeta_k= \xi_{i_k}^\ast(t)$ for $k=1,2$. Hence, \eqref{zeta.stability} follows from the stability of the ODE \eqref{equation orientation one particle}.

 We  estimate, recalling the definition of $D_i$ from \eqref{def.D_i} and using $\1_{\sigma_i = 1} = \tfrac 1 2(1 + \sigma_i)$ and Young's inequality
 \begin{align}
     \E \left[\left(N_{Q}^{1,+} - N_{Q}^{2,+}\right)^2\right] &= \E \Bigg[\bigg(\sum_{(X_i,\xi_i^0) \in Q} \1_{\{\sigma_i = +1\}} - m^+(\xi_i^\ast) \bigg)^2 \Bigg]\\
     & \lesssim N_Q + \sum_{(X_i,\xi_i^0) \in Q} \sum_{\substack{{(X_j,\xi_j^0) \in Q}\\{j \neq i}}} (|\Cov(\sigma_i,\sigma_j)| + D_i D_j) \\
     & \lesssim N_Q +   N_Q \sum_{(X_i,\xi_i^0) \in Q} D_i^2 + \sum_{(X_i,\xi_i^0) \in Q} \sum_{\substack{{(X_j,\xi_j^0) \in Q}\\{j \neq i}}} |\Cov(\sigma_i,\sigma_j)|.
 \end{align}
 Hence,  
 \begin{align} \label{split.E.N_Q} 
 \begin{aligned}
     \frac{1}{N}  \sum_{Q \in \mathcal Q} \E\left[\left| N_{Q}^{1,+} - N_{Q}^{2,+} \right|\right] & \lesssim \frac 1 {N}  \sum_{Q \in \mathcal Q} \left(\E\left[\left| N_{Q}^{1,+} - N_{Q}^{2,+} \right|^2\right] \right)^{1/2} \\
     &\lesssim \frac 1 {N}\sum_{Q \in \mathcal Q}  N_Q^{1/2} 
     +  \frac 1 {N}\sum_{Q \in \mathcal Q} \bigg(N_Q \sum_{(X_i,\xi_i^0) \in Q} D_i^2\bigg)^{1/2}  \\
     & + \frac 1 {N}\sum_{Q \in \mathcal Q} \bigg( \sum_{(X_i,\xi_i^0) \in Q} \sum_{\substack{{(X_j,\xi_j^0) \in Q}\\{j \neq i}}} |\Cov(\sigma_i,\sigma_j)| \bigg)^{1/2}.
     \end{aligned}
 \end{align}
 Thanks to  \eqref{number.cubes}, and $\sum_Q N_Q = N$, the first right-hand side term can bounded by
\begin{align} \label{sum.N_Q}
    \frac 1 {N}\sum_{Q \in \mathcal Q}  N_Q^{1/2}
    \lesssim \frac 1 {N^{1/2} \delta^{5/2}} .
\end{align}
Moreover, by \eqref{control wasserstein microscopic p_eps},
\begin{align}  \label{sum.D_i}
\begin{aligned}
     \frac 1 {N}\sum_{Q \in \mathcal Q} \bigg(N_Q \sum_{(X_i,\xi_i^0) \in Q} D_i^2\bigg)^{1/2} &\lesssim \frac 1 N \bigg(\sum_{Q \in \mathcal Q} N_Q\bigg)^{1/2} \bigg(\sum_{Q \in \mathcal Q} \sum_{(X_i,\xi_i^0) \in Q} D_i^2 \bigg)^{1/2} \\
     &= \frac 1 {\sqrt N} \bigg(\sum_i D_i^2 \bigg)^{1/2} \lesssim e^{\tfrac{-ct}\eps}    + \eps^{1/2} + \phi.
     \end{aligned}
\end{align}
Furthermore, by \eqref{est.Cov} and \eqref{control wasserstein microscopic} we have
\begin{align}
    &\sum_{Q \in \mathcal Q}  N_Q^{-1} \sum_{(X_i,\xi_i^0) \in Q} \sum_{\substack{{(X_j,\xi_j^0) \in Q}\\{j \neq i}}} |\Cov(\sigma_i,\sigma_j)|^2 \\
    &\quad \lesssim  \sum_{Q \in \mathcal Q}  N_Q^{-1} \sum_{(X_i,\xi_i^0) \in Q} \sum_{\substack{{(X_j,\xi_j^0) \in Q}\\{j \neq i}}}  \bigg(e^{\tfrac{-ct}\eps} \\
    &\qquad \qquad +  \frac 1 \eps \int_0^t e^{\tfrac{-c(t-s)}\eps}\Big(\E[|\xi_i(s) - \xi_i^\ast(s)|^2] + \E[|\xi_j(s)- \xi_j^\ast(s)|^2]\Big) \dd s \bigg)\\
    & \quad \lesssim N e^{\tfrac{-ct}\eps} +  \frac 1 \eps \int_0^t e^{\tfrac{-c(t-s)}\eps} \sum_i \E[|\xi_i(s) - \xi_i^\ast(s)|^2]  \dd s  \\
    & \quad \lesssim N e^{\tfrac{-ct}\eps} + N(\eps + \phi^2).
\end{align}
Therefore,
\begin{align} \label{sum.Cov}
\begin{aligned}
   & \frac 1 {N}\sum_{Q \in \mathcal Q} \bigg(\sum_{(X_i,\xi_i^0) \in Q}  \sum_{\substack{{(X_j,\xi_j^0) \in Q}\\{j \neq i}}}  |\Cov(\sigma_i,\sigma_j)| \bigg)^{1/2}\\ 
   &\quad \lesssim \frac 1 {N}\sum_{Q \in \mathcal Q} N_Q^{1/2} \bigg(  \sum_{(X_i,\xi_i^0) \in Q} \sum_{\substack{{(X_j,\xi_j^0) \in Q}\\{j \neq i}}}   |\Cov(\sigma_i,\sigma_j)|^2  \bigg)^{1/4} \\
    & \quad \lesssim \frac 1 {N} \bigg(\sum_{Q \in \mathcal Q}  N_Q \bigg)^{3/4} \bigg(  N_Q^{-1} \sum_{Q \in \mathcal Q} \sum_{(X_i,\xi_i^0) \in Q}  |\Cov(\sigma_i,\sigma_j)|^2  \bigg)^{1/4} \\
    & \quad \lesssim e^{\tfrac{-ct}\eps} + \eps^{1/4}  + \phi^{1/2}.
    \end{aligned}
\end{align}
Inserting estimates \eqref{sum.N_Q}, \eqref{sum.D_i} and \eqref{sum.Cov} into \eqref{split.E.N_Q} yields
\begin{align}
    \frac{1}{N}  \sum_{Q \in \mathcal Q} \E\left[\left| N_{Q}^{1,+} - N_{Q}^{2,+} \right|\right] \lesssim  \frac 1 {N^{1/2} \delta^{5/2}}  +  e^{\tfrac{-ct}\eps} + \eps^{1/4}  + \phi^{1/2}.
\end{align}
In combination with \eqref{W_2.h_N.bar} we deduce
\begin{align}
        \E\left[W_2^2( h_N^1, h_N^2 )\right] & \lesssim \delta^2 +  \frac 1 {N^{1/2} \delta^{5/2}}  +  e^{\tfrac{-ct}\eps} + \eps^{1/4}  + \phi^{1/2}.
\end{align}
Finally, we choose $\delta = N^{-1/9}$ yielding
\begin{align} \label{W_2.h^N_1,2.final}
        \E\left[W_2^2( h_N^1, h_N^2 )\right] & \lesssim N^{-2/9} +  e^{\tfrac{-ct}\eps} + \eps^{1/4}  + \phi^{1/2}.
\end{align}
\\[2mm]
\textbf{Substep 2.3:} \emph{Conclusion.} \\
Inserting estimates \eqref{W_2.h_N^1}, \eqref{W_2.h_N^2} and \eqref{W_2.h^N_1,2.final} into \eqref{split.W_2.h_N} yields \eqref{h_N,h}.\\[2mm]
\textbf{Step 3:} \emph{Proof of \eqref{u_N.u}.} \\
We estimate
\begin{align}
    \|u_N - u\|_{L^p(B_1(x))} \leq \|u_N - \sum_j v_j\|_{L^p(B_1(x))} &+ \|\sum_j (v_j -   \tilde v_j)\|_{L^p(B_1(x))} \\
    &+ \|\sum_j \tilde v_j - u\|_{L^p(B_1(x))} .
\end{align}
Here, the functions $v_j$ and $\tilde v_j$ are defined in \eqref{v_j} and \eqref{tilde.v_j} with $(\bar \xi, \bar \sigma, \bar H) = (\xi(t), \sigma(t), H(t,\cdot))$.
Since $p \in (1,3/2)$, we can combine Lemma \ref{lemma:reflection method} and Sobolev embedding to
\begin{align} \label{sum.v_ju_N}
    \|u_N - \sum_j v_j\|_{L^p(B_1(x))} \lesssim \phi^{3/2}.
\end{align}
Moreover, Lemma \ref{lem:v_j.tilde.v_j} gives
\begin{align} \label{sum.v_j.tilde.v_j}
    \|\sum_j (v_j -   \tilde v_j)\|_{L^p(B_1(x))} \lesssim N r^{1 + 3/p} = \phi r^{3/p - 2}.
\end{align}
We observe that $\tilde u_N := \sum_j \tilde v_j$ satisfies
\begin{align}
\left\{ \begin{array}{rl}
       \begin{aligned}
           -\Delta \tilde u + \nabla \tilde q = &\phi \int_{\S^2 \times \spinstate}   \varsigma  \zeta \cdot \nabla H   h_N(t,x, \d \zeta, \dd \varsigma)  
       \\
       &+ \phi \dv \int_{\S^2 \times \spinstate} \varsigma  (\Id + \mR) \zeta \wedge H h_N(t,x, \d \zeta, \dd \varsigma)  
       \end{aligned} & \quad \text{in } \R^3,\\
        \dv \tilde u = 0 & \quad \text{in } \R^3,
    \end{array}\right.
\end{align}
and that $u$ satisfies 
\begin{align}
\left\{ \begin{array}{rl}
       \begin{aligned}
        -\Delta  u + \nabla \tilde q = &\phi \int_{\S^2 \times \spinstate}   \varsigma  \zeta \cdot \nabla H   h(t,x, \d \zeta, \dd \varsigma) \\
        &+ \phi \dv \int_{\S^2 \times \spinstate} \varsigma  (\Id + \mR) \zeta \wedge H h(t,x, \d \zeta, \dd \varsigma)  \end{aligned} & \quad \text{in } \R^3, \\
        \dv u = 0 & \quad \text{in } \R^3. 
        \end{array}\right.
\end{align}
 Let $g \in L^{p'}$ with $\supp g \subset \overline{B_1(x)}$. Let $\varphi \in \dot H^1(\R^3)$ be the solution to 
\begin{align}
    -   \Delta  \varphi + \nabla  p = g,  \quad \dv  \varphi  = 0 \quad \text{in }  \R^3.
 \end{align}
Let $\gamma \in \mathcal P((\R^3 \times \S^2 \times \spinstate)^2)$ be an optimal plan for $(h_N(t,\cdot),h(t,\cdot))$. Then, 
\begin{align}
    \int_{\R^3} g \cdot (u - \tilde u) = \phi \int_{(\R^3 \times \S^2 \times \spinstate)^2} \Psi(x_1,\zeta_1,\varsigma_1) - \Psi(x_2,\zeta_2,\varsigma_2) \dd \gamma,
\end{align}
where
\begin{align}
    \Psi(x,\zeta,\varsigma) = \varsigma  v(x)  \cdot (  \zeta \cdot \nabla H(x)) - \varsigma \nabla v(x) :  (\Id + \mR(\zeta)) \zeta \wedge  H(x) .
\end{align}
 By regularity theory for the Stokes equation and Morrey's inequality,
\begin{align}
    [\Psi(x,\zeta,\varsigma)]_{1 - 3/p'} \lesssim \|v\|_{C^{1,1 - 3/p'}} \lesssim \|g\|_{L^{p'} \cap L^1} \lesssim \|g\|_{L^{p'}},
\end{align}
where we used $\supp g \subset \overline{B_1(x)}$ in the last estimate.
Hence,
\begin{align}
    \int_{\R^3} g \cdot (u - \tilde u) &\lesssim \phi \|g\|_{L^{p'}} \int_{(\R^3 \times \S^2 \times \spinstate)^2} |(x_1,\zeta_1,\varsigma_1) - (x_2,\zeta_2,\varsigma_2)|^{1 - 3/p'} \dd \gamma \\
    &\lesssim \phi  \|g\|_{L^{p'}} W^{1 - 3/p'}_2(h_N,h) .
\end{align}
Since $g \in L^{p'}$ with $\supp g \subset \overline{B_1(x)}$ was arbitrary, this yields 
\begin{align}
    \|\sum_j \tilde v_j - u\|_{L^p(B_1(x))} = \|\tilde u - u\|_{L^p(B_1(x))} \lesssim \phi W^{1 - 3/p'}_2(h_N,h).
\end{align}
Combining this estimate with \eqref{sum.v_j.tilde.v_j} and \eqref{sum.v_ju_N} yields \eqref{u_N.u}.
\end{proof}

\appendix 

\section{Modeling and nondimensionalization} \label{sec:app}

In this section we explain in more detail the modeling, simplifications and nondimensionalization that leads to the microscopic model \eqref{u_N}--\eqref{ODE.Xi}.
We consider the microscopic system with physical dimensions.

We assume that the particles are single domain  magnets, that is the magnetization is uniform in the particles and equal to the saturation magnetization. This is an idealization of more complex models where the magnetization may vary inside of the particles (see \cite{BedantaPetrcicKleeman15}). We assume that for each particle the  constant magnetization can only take two values, namely in the direction of the so-called easy axis which arises from the crystalline structure of the particles and/or their elongated shape. 


With the notation from Section \ref{subsection Microscopic problem}, we denote by $\check X_i, \check \xi_i, \check \sigma_i$ the particle positions, orientations and (super)spins, and by $\check r$ their common radius. Then,  the space occupied by the $i$-th particle is
\begin{align}
    \check \B_i = \check X_i + \check r R_{\check \xi_i} \B.
\end{align}

Given an external magnetic field $H_{ex}$, the total electric and magnetic field is governed by Maxwell's equations. However, it is typical for ferrofluids to neglect the effect of the electric field, which we do in the following (cf. \cite[Chapter 3.4]{Rosensweig02}). We denote $H_{in}$ the internal magnetic field, which thus obeys
\begin{align}
    \left\{ \begin{array}{rl}
    \dv (H_{in} + M) = 0, &\qquad \text{in } (0,T) \times \Omega, \\
    \curl  H_{in} = 0 &\qquad \text{in } (0,T) \times \Omega,
    \end{array} \right.
\end{align}
where $M$ is the magnetization given as
\begin{align}
    M = m_{s}  \sum_i  \check \sigma_i \check \xi_i \1_{\check \B_i}.
\end{align}
 where $m_s$ is the saturation magnetization of the nanoparticle material.
 
 The magnetic flux density $B$ is related to $\check H = H_{ex} + H_{in}$ via
 \begin{align}
     B = \mu_0 (\check H + M),
 \end{align}
where $\mu_0$ is the vacuum permeability. 
Since there are no free charges or currents, the Lorentz force density is given as 
\begin{align}
    f = \curl M \times B =  B \cdot \nabla M - (\nabla M)^T B,
\end{align}
which is supported on $\cup_i \partial \check \B_i$.
Hence, applying the divergence theorem and using that $\nabla M =0$ in $\check \mB_i$ and $\dv \check H =0$ (this also holds for the external magnetic field, because magnetic monopoles do not exist), the magnetic force and torque acting on the $i$-th particle are
\begin{align}
    F_i &= \int_{\partial \check \mB_i} f  =  \mu_0  \int_{\check \mB_i}\nabla (M \cdot \check H), \\
    T_i &=  \int_{\partial \check B_i} (x - X_i) \times f = \mu_0 \int_{\check \mB_i}  M \times \check H + (x - X_i) \times \nabla( M \cdot \check H).
\end{align}

The fluid velocity $\check u_N$ is given as the unique solution in $\dot H^1(\R^3)$ to  the Stokes equations with no slip boundary conditions at the particles. We assume that the particles are inertialess, i.e. instead of prescribing the particle translational and angular velocities, we require  that the magnetic forces and torques are balanced by the forces and torques caused by the fluid. Hence, 
\begin{align} \label{check.u}
\left\{\begin{array}{l}
    -   \nu \Delta \check u_N + \nabla \check p_N = 0,  \quad \dv \check u_N^{\check \sigma, \check \xi} = 0 \quad \text{in }   \R^3 \setminus \cup_i \check \B_i, \\
     D \check u_N = 0 \quad \text{in }   \cup_i \check \B_i, \\
        \forall i \quad -\int_{\partial \check \B_i} \Sigma_\nu[\check u_N,\check p_N] n \dd x  =  \check \sigma_i \mu_0 m_s \int_{\check \B_i}    \nabla (\check \xi_i \cdot \check H) \dd x,  \\ \forall i \quad 
    -\int_{\partial \check \B_i} (x-X_i)   \times \Sigma_\nu[\check u_N,\check p_N] n  \dd x  =  \check \sigma_{i} \mu_0 m_s   \int_{\check \B_i}    \check \xi_i  \times  \check H +  (x - \check X_i) \times  \nabla ( \check \xi_i \cdot \check H) \dd x.
    \end{array} \right. 
\end{align} 
where  $\nu$ is the fluid viscosity and $\Sigma_\nu[u,p] = 2 \nu D u  - p \Id$  the corresponding stress tensor. Here $D u = \tfrac 1 2 (\nabla u + (\nabla u)^T)$ is  the symmetric gradient of $u$ and it is classical that $\check D u = 0$ in $\mathcal \mB_i$ is equivalent to the existence of $V_i, \Omega_i$ such that $u(x) = V_i +  \Omega_i \times (x - X_i) $ in $\mB_i$. 
The particles' evolution is governed by these rigid body velocities, i.e.
\begin{align}
            \frac{\d}{\d t} \check X_i(t) &= V_i(t) =   \check u_N(t,\check X_i) , \\
              \frac{\d}{\d t} \check \xi_i(t) &= \Omega_i(t) \times \check \xi_i(t) =  \frac 1 2 \curl \check u_N(t,\check X_i) \times  \check \xi_i(t) =\nabla \check u_N(t,\check X_i) \check \xi_i(t).
\end{align}

The spin flips occur with rate
\begin{align} \label{check.lambda}
\begin{aligned}
    \check \lambda_i(t, \check X, \check \xi(t), \check \sigma(t)) \dd t =\frac 1 {\tau_0} \exp\left( \frac {\check r^3 |\mB| K  + \check \sigma_i(t) \mu_0 m_s\int_{\check \mB_i(t)}  \check H(t,x) \cdot \check \xi_i(t) }{k_B \Theta} \right),
    \end{aligned}
\end{align}
where $K$ is the effective anisotropy energy density, $k_B$ is the Boltzmann constant, $\Theta$ the absolute temperature  and $\tau_0$ is the elementary spin flip time (see \cite{BedantaPetrcicKleeman15}).
The spin flip model is an approximation for a continuous time evolution for the magnetization $M_i$ of the $i$-th particle based on the stochastic Landau-Lifschitz-Gilbert equation. We refer to \cite{StochasticFerromagnetism} for the analysis of more complex stochastic ferromagnetism models.

\paragraph{Simplifications}
Note that we neglect thermal noise regarding the particle positions, orientations and the fluid velocity, even though the ferromagnetic nanoparticles are so small that thermal noise is expected to be important regarding the particle orientations. 
Both the effective Brownian motion for the particle orientation and the spin flip are mechanisms that are supposed to result in relaxation phenomena that lead to superparamagnetic behavior: a strong magnetization in an externally applied magnetic field that quickly fades when the external field is removed.
According to \cite{Rosensweig87}, the two relaxation mechanisms are equally important for particles of around $10nm$. 

As a further simplification, we do not take into account the particle translations and we neglect the internal magnetic field, i.e. we assume $\check H= H_{ex}$, i.e. $\check H$ is a given function. In fact, it is believed that the interaction of the particles through the internal magnetic field, is negligible (see e.g. \cite{BedantaPetrcicKleeman15}). From the mathematical point of view, this singular interaction might not be easy to deal with but we do not investigate this aspect in the present work.
For purely notational convenience, we pretend that the magnetic field $\check H$ was constant inside of the particles. This has the advantage that we can replace the integrals involving $\check H$ in \eqref{check.u} and \eqref{check.lambda} and that the second right hand side term in the torque vanishes (assuming a centered reference particle, i.e. $\int_{\mB}x \dd x = 0$).

With these simplifications, our model becomes
\begin{align} \label{check.full}
\left\{\begin{array}{rl}
    -   \nu \Delta \check u_N^{\check \sigma, \check \xi} + \nabla \check p_N^{\check \sigma, \check \xi} = 0,  \quad \dv \check u_N^{\check \sigma, \check \xi} = 0 &\quad \text{in }  \R^3 \setminus \cup_i \check B_i, \\
     D \check u_N^{\check  \sigma, \check  \xi} = 0 &\quad \text{in }   \cup_i \check B_i, \\
        -\int_{\partial \check \B_i} \Sigma_\nu[\check u_N^{\check \sigma, \check \xi},\check p_N^{\check \sigma, \check \xi}] n \dd x  =  \check\sigma_i \mu_0 m_s \check r^3 |\mB|   \nabla \check H(X_i) &\quad \forall i, \\
    -\int_{\partial \check \B_i} (x-\check X_i)   \times \Sigma_\nu[\check u_N^{\check \sigma,\check  \xi},\check p_N^{\check \sigma, \check \xi}] n  \dd x  =  \check \sigma_{i} \mu_0 m_s  \check r^3 |\mB|  \check \xi_i  \times  \check H(X_i) &\quad \forall i, \\
            \check X_i(t) =   X_i(0)  &\quad \forall i, \\
              \frac{\d}{\d t} \check \xi_i(t) =  \frac 1 2 \curl \check u_N^{\check \sigma, \check \xi}(t,X_i) \times  \check \xi_i(t)  &\quad \forall i,\\
    \check \lambda_i(t, \check X, \check \xi(t), \check \sigma(t)) \dd t =  \lambda^{\check \sigma_i(t)}(t, \check X_i,\check \xi_i(t)) \dd t  &\quad \forall i, \\
    \check \lambda^{\pm}(t,x,\zeta) =\frac 1 {\tau_0} \exp\left( \frac {\check r^3 |\mB| (K  \pm  \mu_0 m_s \check H(t,x) \cdot \zeta )}{k_B \Theta} \right).
        \end{array} \right. 
\end{align}



\paragraph{Nondimensionalization}

 We non-dimensonalize the above model in terms of a typical length scale $L > 0$ (the size of the cloud of magnetic nanoparticles)  and a timescale $\bar T > 0$ which is chosen such that the typical angular velocity of the particles is of order $1$. The angular velocity of an isolated particle in a fluid on which a torque $T$ acts is (cf. Section \ref{sec:single}) 
 \begin{align}
      |\omega| \sim \frac{|T|}{\nu \check r^3}.
 \end{align}
  With $\bar H > 0$ the typical magnitude of the external magnetic field, we therefore choose
  \begin{align}
      \bar T = \frac{\nu}{\mu_0 m_s \bar H}.
  \end{align}   
  We then define
  \begin{align}
            r = \frac{\check r}{L}, && \mB_i(t) = L \check \mB_i(\bar T t),  &&     \xi_i(t) = \check \xi_i(\bar T t), &&
    X_i(t) = \frac{\check X_i(0)}{L}, &&
    \sigma_i(t) = \check \sigma_i(\bar T t), 
  \end{align}
      as well as
      \begin{align}
      H(t,x) = \frac{\check H(\bar T t, L x)}{\bar H}, && 
            u_N(t,x) = \frac{\bar T}{L} \check u_N(\bar T t, L x),  && p_N(t,x) = \frac{\bar T}{\nu} \check p_N(\bar T t, L x)
  \end{align}
  and
  \begin{align}
      \lambda_i(t,  X, \xi(t),  \sigma(t)) = {\eps} \bar T \check \lambda_i(\bar T t, L \check X, \check \xi(\bar T t), \check \sigma(T t)) , &&  
     \lambda^{\pm}(t,x,\zeta)  =  \eps \bar T \check \lambda^{\pm}(\bar T t,L x,\zeta),
  \end{align}
  with 
  \begin{align}
      \eps = \frac{\tau_0}{\bar T} \exp\left( - \frac {\check r^3 |\mB| K}{k_B \Theta} \right).
  \end{align}
Then, \eqref{check.full} turns into  \eqref{u_N}--\eqref{ODE.Xi} with
\begin{align}
    b = \frac{\check r^3 |\mB| \mu_0 m_s \bar H}{k_B \Theta}.
\end{align}

\paragraph{Typical values}
(In the following $\mathrm s, \mathrm m, \mathrm K, \mathrm A$ refer to the SI units seconds, meter, Kelvin, Ampere.) We consider a fluid with a viscosity $ \nu \approx 10^{-3} \mathrm{kg}/(\mathrm{ms})$, magnetite nanoparticles with $K \approx 10^3 \mathrm J / \mathrm{m^3}$, $m_s \approx 10^6 \mathrm A / \mathrm m$, $\check r \approx 10^{-8} \mathrm m$ which are typical values (see e.g. \cite{Hadadian22}).  We observe that at standard temperatures $\Theta \approx 300 \mathrm K$
we have 
\begin{align}
    \bar T \approx 10^{-3} \frac 1 {\bar H} \frac{\mathrm{A s}}{\mathrm m}.
\end{align}
Moreover, 
\begin{align}
    b \approx 10^{-3}  \bar H \frac{\mathrm m}{\mathrm A}.
\end{align}
Finally, 
\begin{align}
    \eps \approx 10^{-6} \bar H \frac{\mathrm m}{\mathrm A} \exp\left( - \check r^3 |\mB| 10^{24}  \mathrm m ^{-3} \right)
\end{align}
(up to some prefactor of order $1$ in the exponential). This means that indeed at around $\check r = 10^{-8}  \mathrm{m}$ (not much dependent on $\bar H$) there is a sharp transition from $\eps \ll 1$ to $\eps \gg 1$.

The earth magnetic field is of order $\bar H \sim 50 \mathrm A/ \mathrm m $. Magnetic fields in applications are typically considerably larger. Therefore, for most practical purposes $\bar T \ll 1 \mathrm s$.

\medskip

These numbers indicates that $\eps \ll \bar T / T_{obs} \ll 1$ with a macroscopic observation time $T_{obs}$ would be an interesting scaling regime. In that case one would ecpect the particles to almost instantaneously align with the external magnetic field. The analysis of this regime lies outside of the scope of the present paper.

 \begin{refcontext}[sorting=nyt]
\printbibliography
 \end{refcontext}

\end{document}